  \documentclass[12pt]{amsart}
\usepackage[utf8]{inputenc}
\usepackage[OT2,T1]{fontenc}
\DeclareSymbolFont{cyrletters}{OT2}{wncyr}{m}{n}
\usepackage{enumerate, amssymb, amsmath, amsthm}
\usepackage{amscd}
\usepackage{graphicx}
\usepackage{hyperref}
\usepackage{xcolor}
\hypersetup{
    colorlinks=true,
    linkcolor=blue,
    filecolor=blue,      
    urlcolor=blue,
    citecolor=magenta
}
 
\usepackage[a4paper,hmargin=2.5cm,vmargin=2.5cm]{geometry}
\usepackage{marginnote}
\usepackage[foot]{amsaddr}
\usepackage[english]{babel}
\graphicspath{ {images/} }
\usepackage{wrapfig,textcomp}
\usepackage[matrix,arrow,curve]{xy}\usepackage{mathabx}
\usepackage[OT2, T1]{fontenc}
\usepackage{tabularx}

\usepackage[a4paper,hmargin=2.5cm,vmargin=2.5cm]{geometry}
\usepackage[english]{babel}
\graphicspath{ {images/} }
\usepackage{wrapfig}
\usepackage[matrix,arrow,curve]{xy}\usepackage{mathabx}
\usepackage{booktabs}
\usepackage{array}

\newtheorem{prop}{Proposition}[section]
\newtheorem{thrm}[prop]{Theorem}
\newtheorem{lemma}[prop]{Lemma}
\newtheorem{cor}[prop]{Corollary}
\newtheorem{mainthm}{Theorem}

\theoremstyle{definition}
\newtheorem{df}{Definition}
\newtheorem{ex}{Example}
\newtheorem{rmk}{Remark}

\newtheorem{q}{Question}

\renewcommand{\phi}{\varphi}
\renewcommand{\epsilon}{\varepsilon}

\renewcommand{\C}{\mathbb C}
\newcommand{\Z}{\mathbb Z}
\newcommand{\Q}{\mathbb Q}
\newcommand{\R}{\mathbb R}

\renewcommand{\H}{\mathbb H}

\renewcommand{\O}{\mathcal O}

\newcommand{\E}{\mathcal E}

\renewcommand{\i}{\sqrt{-1}}
\renewcommand{\o}{\otimes}
\newcommand{\di}{\partial}

\newcommand{\dibar}{\overline{\partial}}
\newcommand{\im}{\operatorname{im}}
\renewcommand{\ker}{\operatorname{ker}}

\newcommand{\rk}{\operatorname{rk}}

\newcommand{\Supp}{\operatorname{Supp}}

\newcommand{\CP}{\mathbb{C}\mathbf{P}}

\newcommand{\Ker}{\operatorname{Ker}}
\newcommand{\Tot}{\operatorname{Tot}}
\newcommand{\id}{\operatorname{id}}
\newcommand{\Grp}{\operatorname{Gr}_{++}}
\newcommand{\arr}{\xrightarrow}
\newcommand{\eps}{\varepsilon}
\DeclareMathSymbol{\Bl}{\mathalpha}{cyrletters}{"59}

\renewcommand{\U}{\operatorname{U}}

\newcommand{\Kn}{\operatorname{Kn}}
\newcommand{\ev}{\operatorname{ev}}
\newcommand{\vel}{\operatorname{vel}}

\newcommand{\Diff}{\operatorname{Diff}}
\newcommand{\F}{\mathbf F}
\usepackage{eucal}

\begin{document}
\title{Holomorphic families of knots}

\author{Rodion D\'eev}

\author{Vasily Rogov}

\begin{abstract}
Let $(M, [g])$ be a $3$-dimensional conformal manifold. The space of knots $\Kn(M)$ in $M$ is an infinite-dimensional manifold that is  known to carry an almost complex structure. This structure is formally integrable by a result of Brylinski. We study finite dimensional holomorphic submanifolds in $\Kn(M)$. We define an holomorphic family of knots in $(M, [g])$ parametrised by a finite-dimensional complex manifold $(X, I_X)$, and construct several series of examples. We show that the base $(X, I_X)$ is K\"ahler, and if $X$ is compact, it is a projective variety of complex dimension at most $2$. In this case the conformal structure $[g]$ uniquely determines the complex structure $I_X$ and vice versa. We prove that if an holomorphic family of knots in $(M, [g])$ over a compact base $(X, I_X)$ defines a foliation on $\mathbf{S}(TM)$, then $(X, I_X) \simeq \CP^1 \times \CP^1$, the manifold $(M, [g])$ is conformally equivalent to either $S^3$ or $\R\mathbf{P}^3$ with round metric, and all knots are geodesic in some round metric in the class.

MSC classes: 53C28, 32L25.
\end{abstract}

\maketitle
\vspace{-10mm}
\tableofcontents

\newpage

\section{Introduction}

\subsection{Geometry of the space of knots}\label{infinite-subsection}

Let $M$ be a smooth oriented $3$-dimensional manifold. By a {\it knot} we mean 
a smooth embedding $\gamma \colon S^1 \to M$ viewed up to orientation-preserving 
reparametrisation (we often denote both the map $S^1 \to M$ and its image by the 
same letter $\gamma$). The quotient $\Kn(M) := \operatorname{Emb}(S^1, M)/\Diff^{+}(S^1)$ 
is called the {\it space of knots} in $M$. Its topology is the key object of 
study in the theory of Vassiliev invariants (\cite[Chapter~15]{CDM}). Different 
geometric structures on $M$ are known to yield geometric structures on this 
space, viewed as an infinite-dimensional (Fr\'echet) manifold. Their 
correspondence may be summarised as follows:
\begin{table}[h]
\centering
\renewcommand{\arraystretch}{1.6}
\begin{tabular}{>{\raggedright\arraybackslash}p{75mm} | 
                >{\raggedright\arraybackslash}p{75mm}}
\toprule
\textbf{Geometric datum on $M^3$} & \textbf{Geometric datum on $\Kn(M)$} \\
\midrule
volume form $\nu \in \Omega^3(M,\mathbb{R})$ & 
    symplectic form $\omega_\nu$ on $\Kn(M)$ \\[4pt]
surface $Y^2 \subset M^3$ & 
    Lagrangian submanifold $\Kn(Y) \subset \Kn(M)$ \\[4pt]
conformal structure $[g]$ & 
    almost complex structure $I_{[g]}$ \\[4pt]
Riemannian metric $g$ & 
    almost K\"ahler structure on $\Kn(M)$ \\[4pt]
point $m \in M$ & 
    divisor $D_m = \{\gamma \in \Kn(M) \mid m \in \gamma\}$ \\
\bottomrule
\end{tabular}
\label{tab:knot-correspondence}
\end{table}

Surprisingly, the finite-dimensional complex submanifolds of $\Kn(M)$ and the 
geometric structures on $M$ they correspond to appear to have never been studied. 
The present paper aims to fill this gap. In particular, we never actually work 
with infinite-dimensional manifolds, and in fact we exercise this privilege right 
away by deflecting the inquisitive reader's questions (undoubtedly already raised) 
of the kind: ``What does it mean for a $2$-form to have non-degenerate top 
exterior power on an infinite-dimensional manifold?'', ``What do you mean by a 
submanifold of half the infinite dimension?'', or even ``Well, how does one 
commute vector fields up there?''. For formal definitions and constructions of 
some of the correspondences in Table~\ref{tab:knot-correspondence}, 
see~\cite[Chapter~III]{Br2}; the reader will also see, in the course of this 
paper, how all these structures arise on finite-dimensional submanifolds of $M$.

In order to do this, we introduce the notion of a \emph{family of knots in $M$ parametrised by a base $X$}, which morally corresponds to a smooth embedding 
$X \hookrightarrow \Kn(M)$ (see Definition~\ref{definition of a family}). What 
follows in the present section is a purely motivational exposition, which will 
nonetheless be quite useful to keep in mind later on.

The key observation is that a conformal structure on $M$ determines a complex 
structure operator on $\Kn(M)$. Indeed, the tangent space to $\Kn(M)$ at a 
knot $\gamma$ is the space of sections of its normal bundle:
\[
T_{[\gamma]} \Kn(M) = \Gamma(S^1, N_{\gamma/M}).
\]
Fix a conformal structure $[g]$ on $M$. It induces a splitting of the exact 
sequence
\begin{equation} \label{basic short}
0 \to  d\gamma(TS^1) \to TM|_{\gamma} \to N_{\gamma/M} \to 0
\end{equation}
via the identification $N_{\gamma/M} \xrightarrow{\sim} d\gamma(TS^1)^{\perp_{[g]}}$. 
The rank-$2$ vector bundle $N_{\gamma/M}$ then inherits a conformal structure and 
an orientation (recall that both $d\gamma(TS^1)$ and $TM$ are oriented). A 
conformal structure together with an orientation on a real plane is the same 
datum as an isomorphism with $\mathbb{C}$, so this turns $N_{\gamma/M}$ into a 
complex line bundle over $S^1$, and thus gives a complex structure operator on 
the space of sections $\Gamma(S^1, N_{\gamma/M})$.

In simple terms: the complex structure operator acts on a normal vector field 
by rotating it around the knot by angle $\frac{\pi}{2}$ in the positive 
direction. 
 
Pointwise application of this construction leads to an almost complex structure operator $I_{[g]}$ on the space $\Kn(M)$. Brylinski showed that it is formally integrable.

\begin{thrm}[Brylinski, \cite{Br2}\footnote{It is worthy to mention that authors like Lempert \cite{Lemp} and Wang \cite{Wang} ascribe this theorem to Drinfeld and LeBrun, nevertheless citing Brylinski's preprints. This likely stems from someone's false memory, since Prof.~Drinfeld confirmed to the authors that his contribution to the infinite-dimensional K\"ahler geometry was related to the Kirillov--Yuriev space ${\Diff^+}(S^1)/\mathrm{U}(1)$, while the complex structure on the quotient {\it by} the ${\Diff^+}(S^1)$-action was something he was not even aware of.}]\label{formal integrability}
Let $(M, [g])$ be a conformal $3$-manifold. There is a splitting $T\Kn(M) \otimes \C = T^{1,0}\Kn(M) \oplus T^{0,1}\Kn(M)$ such that $I_{[g]}$ acts on them by multiplication by $\i$ and $-\i$, respectively. Moreover,
\[
[T^{0,1}M, T^{0,1}M] \subseteq T^{0,1}M.
\]
\end{thrm}
We refer the reader once again to \cite{Br2} for both the definition of the Lie bracket of vector fields on $\Kn(M)$ and the proof of the theorem.

Brylinski's theorem does not imply that $\Kn(M)$ is a complex manifold in the usual sense of an atlas with holomorphic transition functions, as Newlander--Nirenberg theorem fails in infinite dimensions. Quite the opposite, Lempert proved (\cite{Lemp}) that $I_{g}$ does not come from any structure of a complex analytic Fr\'echet manifold on $\Kn(M)$. Nevertheless, an immediate corollary of Theorem \ref{formal integrability} is the following. Suppose that $X \subset \Kn(M)$ is a finite-dimensional manifold such that $TX \subset T\Kn(M)$ is preserved by $I_{[g]}$. Then $(X, I_{[g]})$ is a complex manifold.

The goal of this paper is to study complex manifolds that arise this way, with special attention to the case where $X$ is compact.

Various other aspects of geometry of $\Kn(M)$ and its close relatives are studied in \cite{Has, MW, LeBrun93, Lemp, Verb, LL, FL}

\subsection{Main results}

First thing that we do in this paper is define finite-dimensional holomorphic families of knots in a conformal manifold $(M, [g])$ without referring to the infinite-dimensional space $\Kn(M)$. Roughly speaking a family of knots is given by a tuple $\mathcal{X}=(X_{\bullet}, X, \pi, \ev)$, where $X$ is a smooth manifold, $\pi \colon X_{\bullet} \to X$ and orientable $S^1$-bundle and $\ev \colon X_{\bullet} \to M$ a smooth map that is injective on the fibers of $\pi$ (\emph{evaluation map}). We also impose some non-degeneracy conditions, see Defintion \ref{definition of a family}. 

For a fixed conformal structure $[g]$ on $M$ and a complex structure $I_X$ on $X$, we define what does it mean for $\mathcal{X}$ to be holomorphic (Definiton \ref{df holomorphic family}). We give multiple examples  of such families and prove several general results about them.

Our first result says that the base of an holomorphic family of knots is always K\"ahler, and projective if compact. This is perharps not so surprising, as the infinite-dimensional space $\Kn(M)$ is itself ``K\"ahler'' in a certain sense (\cite[Proposition 3.6.1]{Br2}).

\begin{mainthm}[Theorem \ref{kaehlerness of families}]\label{main kaehlerness} Let $(M, [g])$ be a conformal $3$-manifold and $(X_{\bullet}, X, I_X, \pi, \ev)$ an holomorphic family of knots in $(M, [g])$. Then $(X, I_X)$ is K\"ahler. If $M$ and $X$ are compact, $(X, I_X)$ is biholomorphic to the analytification of a complex projective variety.
\end{mainthm}

The next theorem is a bit more unexpected and can be thought of as a bound on the dimension of compact holomorphic submanifolds of $\Kn(M)$.

\begin{mainthm}[Theorem \ref{at most two}]\label{main at most two}
Let $(M, [g])$ be a conformal $3$-manifold and $(X_{\bullet}, X, I_X, \pi, \ev)$ an holomorphic family of knots in it. If $X$ is compact, $\dim_{\C}X \le 2$.
\end{mainthm}

In subsection \ref{examples in sphere subsec} we produce many holomorphic families of knots in a round sphere parametrized by complex projective surfaces (i.e. $\dim_{\C} X=2$). It turns out that this situation is quite specific.

First of all, we prove that if $(X, I_X)$ is a compact complex surface parametrising holomorphic family of knots in $(M, [g])$, the complex structure $I_X$ uniquely determines the conformal structure $[g]$ and vice versa (Theorem \ref{two-complex-structures-proposition}).

Second, we show that under some mild assumptions every holomorphic compact family of knots of complex dimension $2$ comes from the family of geodesics on the round sphere. This result can be considered as a converse of theorems of Ono \cite{Ono} and Hitchin \cite{Hit}, as we will know explain. 

Assume that $(M, g)$ is a compact Riemannian $3$-manifold that is \emph{Zoll}, i.e.~all its geodesics are closed and of the same minimal period. Then the geodesic flow determines a foliation with closed orbits $\mathcal{F}$ on the spherisation of the tangent bundle $\mathbf{S}M$ and the quotient $X=\mathbf{S}M/\mathcal{F}$ is a smooth manifold. There is a natural family of knots in $M$ parametrised by $X$ (\emph{a priori} not holomorphic).

\begin{thrm}[Ono, \ {\cite{Ono}}]\label{ono thrm}
If $M$ is a $3$-sphere, $X$ is homeomorphic to $S^2 \times S^2$. 
\end{thrm}

Theorem \ref{ono thrm} essentially covers all geodesic families of knots: it is not hard to show that if a manifold admits Zoll metric, then its fundamental group is either trivial or equals $\Z/2\Z$. By Thurston's uniformisation, $M$ is either a sphere or $\R\mathbf{P}^3$. In the latter case, the space of geodesics is again isomorphic to $S^2 \times S^2$.

Ono's theorem does not speak about the complex structure (we should warn the reader that paper \cite{Ono} also considers certain symplectic form on $X$; this form \textbf{does not} coincide with the K\"ahler form given by Theorem \ref{main kaehlerness}). The complex structure on the space of geodesics appears in the work of Hitchin \cite{Hit}.

\begin{thrm}[Hitchin, {\cite{Hit}}]\label{hitchin thrm}
Let $(M, g)$ be a $3$-dimensional orientable Riemannian manifold. Then the space of geodesics admits a canonical almost complex structure that is integrable if and only if $g$ has constant sectional curvature.
\end{thrm}

The complex structure constructed by Hitchin is the same as the one on the space of knots: it acts by rotating a normal vector field by angle $\frac{\pi}{2}$ in the positive direction. As we mentioned above, the space of geodesics is rarely a smooth manifold, so Hitchin's theorem is to be understood locally: it says that the rotation of a Jacobi vector field by a right angle is again a Jacobi vector field if and only if the sectional curvature is constant.

Theorems \ref{ono thrm} and \ref{hitchin thrm} combined tell us that if the geodesics in $(M,g)$ form an holomorphic family of knots, then $(M, g)$ is conformally equivalent either to $S^3$ or $\R\mathbf{P}^3$ with round metric. Our last main result is that the opposite is also true: if an holomorphic family of knots in a weak sense resembles a family of geodesics, then it is conformally equivalent to the geodesic family on $S^3$ or $\R\mathbf{P}^3$ with round metric. We make the idea of a family ``resembling a family of geodesics'' precise by introducing the following two definitions.

We say that a family is \emph{flow-like} if for every point $p \in M$ and tangent vector $v \in T_pM$ there exists a unique knot $\gamma$ from the family that passes through $p$ tangent to $v$, and {\it real} if for vectors $v$ and $-v$ these knots coincide as subsets of $M$.

\begin{mainthm}[Theorems \ref{flowlike thrm} and \ref{flowlike real thrm}]\label{flowlike thrm main}
Let $\mathcal{X}=(X_{\bullet}, X, \pi, \ev)$ be a compact family of knots in $M$ and $\dim_{\R}X=4$. Suppose that the following conditions hold:
\begin{itemize}
\item[(i)] $\mathcal{X}$ is flow-like;
\item[(ii)] there exists a conformal structure $[g]$ on $M$ and a complex structure $I_X$ on $X$ such that $\mathcal{X}$ is holomorphic.
\end{itemize}
Then $(X, I_X) \simeq \CP^1 \times \CP^1$,  and $M$ is diffeomorphic to the quotient of $S^3$ by a cyclic group. If moreover
\begin{itemize}
\item[(iii)] $\mathcal{X}$ is real,
\end{itemize}
then $M$ is diffeomorphic to $S^3$ or $\R\mathbf{P}^3$, $[g]$ is the conformal class of the round metric, and $X$ is the space of geodesics for some round metric in $[g]$. 

The conclusion fails if one drops any of the conditions \textit{(i)-(iii)}.
\end{mainthm}

\begin{rmk}
There are several ways to generalise our notion of a family of knots.
 \begin{itemize}
\item[1)] The usual notion of a knot can be extended by allowing $\gamma$ to be an immersion with self-intersections; such knots are called \emph{singular knots} in \cite{Br2}. The normal bundle $N_{\gamma/M}$ is still well-defined in this setting, and the complex structure $I_{[g]}$ extends to the space of singular knots $\widehat{\Kn(M)}$, so the notion of a family of knots extends accordingly. Most of our results appear to carry over to this generality. We work with smooth knots throughout for three reasons: firstly, it simplifies the notation considerably and makes the proofs less technical, especially in Sections~\ref{pinned loci subsec} and~\ref{at most two subsec}; secondly, we are not aware of any positive-dimensional holomorphic family in which singular knots genuinely appear; thirdly, this completion of the space of knots is not unique, and for other completions of the space of knots the results of our paper may not hold. Whether positive-dimensional families other than ours exist at all in the completion by singular knots is an interesting open question.

\item[2)] One could also allow the base $X$ of a family to be a singular complex analytic space, which would be natural from the point of view of the intrinsic complex geometry of $\Kn(M)$. Finding a suitable definition of a finite-dimensional family over such a base would, however, require additional work. Such families seem nonetheless worth studying; see e.g.\ Remark~\ref{singular rmk}.

\item[3)] Yet another extension would be to allow $X$ to be a smooth complex orbifold, which would enable one to incorporate the remaining finite quotients of the round sphere into Theorem~\ref{flowlike thrm main}. In a broader direction, if in Hitchin's Theorem~\ref{hitchin thrm} one takes $M$ to be a flat or hyperbolic $3$-manifold, the space of geodesics of $M$ admits a structure of a complex analytic stack. Studying its properties and its relation to $\pi_1(M)$ would be very interesting, though this lies well beyond the scope of the present paper.
\end{itemize}
\end{rmk}


\subsection{Organisation of the paper} In Section \ref{families-section}, we introduce the notions of  families  and holomorphic families of knots. It is in fact equivalent to the notion of a finite-dimensional submanifold (resp. complex submanifold) in the infinite-dimensional space of knots. The only novelty is that we define it in purely finite-dimensional terms. 

In Section \ref{examples}, we several examples of finite-dimensional families of knots. We genuinely give two sources of examples: one coming from projective geometry of the space of circles on the round sphere, the other one is associated with holomorphic line bundles on Riemann surfaces. We also introduce the notions of \emph{real} and \emph{flow-like} families and illustrate them through our examples.

In Section \ref{CR-section}, we remind the reader of LeBrun's CR twistor spaces and Lempert's space of transversal knots in a CR manifold. We use them to give a characterisation of holomorphic families of knots in terms of CR geometry.

In Section \ref{geometry-of-families-section}, we study the geometry of compact complex families of knots. In particular, we prove Theorems \ref{main kaehlerness} and \ref{main at most two}. We also describe the holomorphic subfamilies formed by knots that pass through a given point.  As an application of results of this section we also prove Theorem \ref{two-complex-structures-proposition} that says that if $(X, I_X)$ is the base of a compact $2$-dimensional holomorphic family of knots in $(M, [g])$ then $I_X$ determines $[g]$ and vice versa.

Finally, in Section \ref{flowlike-section}, we conclude with our proof of the converse to Hitchin's theorem (theorem \ref{flowlike thrm main} above).

\paragraph{\bf Acknowledgments.} We would like to thank Anis Bousclet, Joel Fine, and Lev Sou\-kha\-nov for interesting comments on different early drafts of this paper. V.R. thanks the Max-Planck-Institut f\"ur Mathematik in den Naturwissenschaften (Leipzig) for the quality of its working environment, and R.D. also thanks the said institution for its cordial hospitality.

\section{Families of knots}\label{families-section}

\subsection{Families of knots}\label{families basic subsec}

Fix a smooth manifold $M$. Let $X$ be a smooth manifold and $\pi \colon X_{\bullet} \to X$ a locally trivial $S^1$-bundle on $X$. We say that $\pi$ is \emph{orientable} if the vertical line bundle $\ker d\pi$ is trivial. This can always be obtained after replacing $X$ with a $2$-to-$1$ cover. An orientation of an $S^1$-bundle is the choice of the orientation on $\ker d\pi$.

Suppose that $X_{\bullet}$ is endowed with a smooth map $\ev \colon X_{\bullet} \to M$ such that for every fibre $F_x=\pi^{-1}(x)$ the restriction of $\ev$ on $F_x$ is a smooth embedding, i.e. determines a knot $\ev|_{\pi^{-1}(x)} \colon S^1 \to M$. Then one can think about $X_{\bullet} \to X$ as a \emph{family} of knots in $M$ parametrised by the base $X$. In this situation, we denote the knot $\ev(\pi^{-1}(x))$ by $\gamma_x$.  
\[
\xymatrix{
& X_{\bullet} \ar[dl]_{\pi} \ar[dr]^{\ev} &\\
X & &M
}
\]

If $X_{\bullet} \to X$ is a family of knots as above, one can associate with it the \emph{variational map} $\delta \colon TX \to T\Kn(M)$ by the following construction. Let $v \in T_xX$ be a tangent vector. Choose a lifting $\widetilde{v} \in \Gamma(F_x, TX_{\bullet}|_{F_x})$ such that $d\pi(\widetilde{v}_m)=v$ for any $m \in \pi^{-1}(x)$. The differential of the evaluation map sends $\widetilde{v}$ to a section $d\ev(\widetilde{v}) \in \Gamma(F_x, \ev^{*}TM|_{F_x})=\Gamma(S^1, \gamma_x^*TM)$. Let $P \colon \gamma_x^*TM \to N_{\gamma/M} = \gamma_x^*TM/d\gamma(TF_x)$ be the natural projection. 
\begin{prop}\label{delta is well-defined}
The resulting section $P(d\ev(\widetilde{v})) \in \Gamma(S^1, N_{\gamma/M})$ depends on $v$ only. This defines a linear map $\delta_x \colon T_xX \to \Gamma{\left(N_{\gamma_x/M}\right)}, \ v \mapsto P(d\ev(\widetilde{v}))$.
\end{prop}
\begin{proof}
Let $\widetilde{v}$ and $\widetilde{v'}$ be two liftings of $v$. Then $\widetilde{u}=\widetilde{v}-\widetilde{v'}$ is a section of the vertical bundle $\ker d\pi=TF_x$. Therefore $d\ev(\widetilde{u})$ is a section of $d\gamma(TS^1) \subset \gamma^*TM$ and $P(d\ev(\widetilde{u}))$ constantly vanishes. The linearity of $\delta_x$ is immediate.
\end{proof}

\begin{rmk}
The heuristics behind this construction is the following. If $X_{\bullet} \to X$ is a family of knots as above, one might imagine that it gives rise to a \emph{moduli map} $\mathfrak{m} \colon X \to \Kn(M)$. The constructed map $\delta$ should be thus thought of as the differential of this moduli map. There are various ways to make this statement precise, see e.g. \cite[Section 3]{Steff}. 
\end{rmk}

\begin{df}\label{definition of a family}
Let $M$ be a smooth manifold. A \emph{family of knots} on $M$ is a quadruple $\mathcal{X}=(X_{\bullet}, X, \pi, \ev)$, where $X$ is a smooth manifold, $\pi \colon X_{\bullet} \to X$ a locally trivial oriented $S^1$-fibration and $\ev \colon X_{\bullet} \to M$ a smooth map that satisfy the following properties:
\begin{itemize}
\item[(i)] for every fibre $F_x=\pi^{-1}(x)$ the restriction $\gamma_x:=\ev|_{F_x} \colon S^1 \to M$ is a knot;
\item[(ii)] $\gamma_x$ and $\gamma_{x'}$  are distinct\footnote{Two knots are distinct if either they are geometrically distinct, i.e. $\im \gamma_x \neq \im \gamma_{x'}$, or their images coincide, but the knots have different orientation.} for $x\neq x'$ in $X$ (\emph{``the moduli map is injective''});
\item[(iii)] $\delta_x \colon T_xX \to \Gamma{\left(\gamma_x,N_{\gamma_x/M}\right)}$ is injective for every $x$ (\emph{``the moduli map is an immersion''}).
\end{itemize}
We refer to $X$ as the \emph{base} of the family $\mathcal{X}$ (alternatively, we say that the family $\mathcal{X}$ is \emph{parametrised} by $X$). We say that a family is \emph{compact} if $X$ is compact.
\end{df}

\subsection{Holomorphic families of knots}\label{holomorphic families subsec}

\begin{df}\label{df holomorphic family}
Let $\mathcal{X}=(X_{\bullet}, X, \pi, \ev)$ be a family of knots on a $3$-dimensional manifold $M$. Let $[g]$ be a conformal structure on $M$. Recall that for every knot $\gamma \subset M$ this defines a complex structure operator $I_{[g]}$ on its normal bundle $N_{\gamma/M}$. Fix $I_X$  an integrable complex structure operator on $X$. We say that the family $\mathcal{X}$ is $I_X$-\emph{holomorphic} if  one has
\[
\delta_x(I_Xv) = I_g\delta_x(v)
\]
for each $x \in X$ and $v \in T_xX$.
\end{df}
 More explicitly, the holomorphicity condition means the following. Let $x \in X$ and $v\in T_xX$. There exists a unique vector field $L(v)\in \Gamma(F_x, TX_{\bullet}|_{F_x})$ such that $d\pi(L(v))=v$ and $d\ev(L(v))$ is $g$-orthogonal to $\dot{\gamma}_x$. Then the family is holomorphic if and only if 
 \[
 \dot{\gamma_x} \times_g d\ev(L(v)) =d\ev(L(I_Xv)).
 \]

 \begin{rmk}
Since the map $\delta_x$ is required to be injective at every $x$, the complex structure $I_X$ is uniquely recovered from $I_{[g]}$. The definition of an holomorphic family makes sense even if $I_X$ is merely an almost complex structure on $X$, but Brylinski's Theorem \ref{formal integrability} forces it to be \textit{a posteriori} integrable.
\end{rmk}

The following proposition is immediate.

\begin{prop}\label{subfamilies}
Let $(X, X_{\bullet}, \pi, \ev)$ be an holomorphic family of knots in a Riemannian $3$-manifold $(M, g)$. Let $Y \subseteq X$ be a smooth holomorphic submanifold. Then $$\left(Y, Y_{\bullet}:=\pi^{-1}(Y), \pi|_{Y_{\bullet}}, \ev|_{Y_{\bullet}}\right)$$ is an holomorphic family of knots in $(M, g)$.
\end{prop}

More generally, an holomorphic family over a base $X$ can be pullbacked along an injective holomorphic map $Y \to X$.

Let us give the first example of an holomorphic family of knots. 

\begin{ex}\label{Hopf}
    Let $(S^3, g)$ be the $3$-dimensional sphere with its round metric and $h \colon S^3 \to \CP^1$ be the Hopf fibration. This determines an holomorphic family of circles inside $(S^3, g)$ parametrised by the complex projective line. More explicitly, $\mathcal{X}_{\operatorname{Hopf}}:=\{S^3, \CP^1, h, \id_{S^3}\}$ is a compact holomorphic family of knots in $(S^3, g)$ of complex dimension $1$.
\end{ex}

In the Section \ref{examples} we provide more examples of holomorphic families of knots. In particular, the reader will see that the example above can be generalised in two different ways: first, this is a particular case of an holomorphic 1-dimensional family of knots associated with an holomorphic Hermitian bundle on a Riemann surface; second, this is a subfamily of a bigger holomorphic family of circles on a conformally flat sphere.

\section{Examples}\label{examples}

In this section, we present several examples of holomorphic families of knots. We also introduce two important properties of families of knots -- \emph{realness} and \emph{flow-likeness} and give examples and non-examples of families sharing these properties.

\subsection{Families of circles in conformally flat $S^3$}\label{examples in sphere subsec}

\begin{prop}
    Consider the $3$-sphere with its round metric $(S^3, g_S)$. All the round circles constitute an holomorphic family. 
\end{prop}
\begin{proof}
    Assume that $\gamma_t$ is a family of round circles in $S^3$ smoothly parametrised by $t \in ]{-\epsilon};\epsilon[$. Let $v$ be the vector in $T_{[\gamma_0]}\Kn(S^3, g_S)$ tangent to it. We can find a conformal transformation $f \colon S^3 \to S^3$ that fixes the circle $\gamma_0$ and acts on $N_{\gamma_0/M}$ as a rotation by the angle $\pi/2$ in the positive direction. Then the vector $I_gv$ is tangent to a new family $\gamma'_t:=f(\gamma_t)$. Therefore, the set of normal vector fields on $\gamma_0$ tangential to families of circles is preserved by $I_g$.
\end{proof}

A simple dimension count shows that the real dimension of the space of circles in $S^3$ equals $6$. In fact, it carries a structure of a semialgebraic domain in a complex projective $3$-fold, as we explain below.

Let  $V$ is a finite-dimenisonal real vector space equipped with a non-degenerate symmetric product $b(\cdot,\cdot)$. The Grassmannian of oriented positive definite $2$-planes $\Grp(V)$ is identified with an open domain 
\[
U:=\mathbb{P}(\{v \in V \o \C \ | \ b(v,v)=0, \ b(v, \overline{v}) >0\}).
\]
inside the quadric $Q=\{b(v,v)=0\} \subset \mathbb{P}(V \o \C)$. The isomorphism is given by sending a line $\C\cdot v \in U$ to $\operatorname{Span}(v, \overline{v})\cap V_{\R}$.

Let $(V,b)=\R^{4,1}$ be a Minkowski space. The conformally round sphere can be represented as the projectivisation of the null-cone $\mathcal{C}=\{b(v,v)=0\}\subset V$. The circles in $S^3$ then identify with the $3$-dimensional subspaces $W \subset V$ of signature $(2,1)$. Replacing $W$ with $W^{\perp}$ we parametrise the space of round circles by $\Grp(\R^{4,1})$.

Upon a point $[v]$ in $U$ one associates the circle in $S^3 =\mathbb{P}(\{b(w,w)=0\})$ cut out by the plane $\langle \operatorname{Re} v, \operatorname{Im} v \rangle$. This induces an orientable $S^1$-bundle $\pi \colon U_{\bullet} \to U$. In other words, we came to the following conclusion.

\begin{prop}\label{all circles example}
Let $b$ be a bilinear non-degenerate form of signature $(4,1)$. Let
\[
U:=\mathbb{P}(\{v \in \C^5 \ | b_{\C}(v,v)=0, \ b_{\C}(v, \overline{v}) >0\}).
\]
There exists an holomorphic family of knots in the round sphere $(S^3, g_S)$ of the form $\mathcal{U}:=(U_{\bullet}, U, \pi, \ev)$ with the property that the map $\ev$ sends every fibre of $U_{\bullet} \to U$ to a circle. Vice versa, every circle in $(S^3, g_S)$ appears in this family.
\end{prop}

\begin{cor}\label{subfamilies of circles}
In the notations as above, let $Q=\mathbb{P}(\{b_{\C}(v,v)=0\})$ be the quadric in $\CP^4$ determined by $b$, so that $U\subset Q$ is its open subset. Let $X \subset Q$ be a complex submanifold that is contained inside $U$. Then there is a family of knots $(X_{\bullet}, X, \pi, \ev)$ in $S^3$. 
\end{cor}

Corollary \ref{subfamilies of circles} can be illustrated by the following important examples. Among all circles in $S^3$ there is a distinguished family of great circles (geodesics in the round metric). They correspond to the $(2,1)$-subspaces containing a line spanned by a fixed vector inside the negative cone. The orthogonals to such subspaces are all contained in a fixed subspace of signature $(4,0)$. They are therefore parametrised by the locus of points in $Q$ cut out by a hyperplane $H \subset \CP^4$. In particular, we realise the union of all great circles in $S^3$ as an holomorphic family over
\[
H \cap Q \simeq \CP^1 \times \CP^1.
\]

A different choice of a metric on $S^3$ inside the round conformal class corresponds to a different choice of a line in the negative cone in $\R^{4,1}$. The group of conformal automorphisms of $S^3$ acts transitively on the projectivisation of the interior of the negative cone, therefore every two such families differ by a conformal automorphism.

\begin{rmk}
This construction is very classical and goes back to the work of Laguerre from 1880s, although its presentation with the null cone in the Minkowski space is undoubtedly not as old, see \cite{Baird,Heller}. This space not just extends the space of great circles on $S^3$ but unifies the spaces of geodesics in all the three simply connected geometries of constant curvature. The other two, the spaces of geodesics in Euclidean $\R^3$ and hyperbolic $\H^3$, show up as surfaces in $U$ as follows. $\R^3$ is conformally equivalent to $S^3 \setminus \{\infty\}$, so the family of straight lines is biholomorphic to the space of circles on $S^3$ passing through a given point, while $\H^3$ is conformally equivalent to the interior of a ball, so the family of hyperbolic geodesics is biholomorphic to the space of circles perpendicular to a given round sphere $S^2 \subset S^3$. As hyperplane sections, they correspond to the case when $H$ is a real hyperplane with degenerate restriction of $b$, and a hyperplane of signature $(3,1)$, respectively.
\end{rmk}

\begin{prop}\label{surfaces of circles}
Let $d>0$. There exists a degree $d$ hypersurface $V_d$ in $\CP^4$ such that the intersection $X_d:=V_d \cap Q$ is smooth and contained inside $U\subset Q$. In particular, there is a compact holomorphic family of knots in the conformally round $S^3$ parametrised by $X_d$.
\end{prop}

\begin{proof}
For $d=1$ one can take a hyperplane $H\subset \CP^4$ such that $H \cap Q$ is the family of great circles for the round metric.

Let $H$ be given by a linear equation $\left\{\sum_{i=0}^{4} c_ix_i=0\right\}$.  Consider the degree $d$ (degenerate) hypersurface $V^{0}_d:=\left(\sum_{i=0}^{4}c_ix_i\right)^{d}=0$. Set-theoretically, $V_d^{0}\cap Q = H \cap Q$. By Bertini lemma, one can perturb it inside the space of degree $d$ hypersurfaces and obtain a family of hypersurfaces $V_d^{\epsilon}$ such that $X_d^{\epsilon}=V_d^{\epsilon} \cap Q$ is smooth for a general value of $\epsilon$.  At the same time, $U$ is open, therefore for a sufficiently small perturbation the variety $X^{\epsilon}_d$ is still contained in $U$.
\end{proof}

Interestingly, taking $d=3$ one obtains holomorphic families of knots in the round $3$-sphere parametrised by K3-surfaces.

 The infinite-dimensional space of all knots comes equipped with an anti-holomorphic involution $\gamma \mapsto \bar{\gamma}$ (switching the orientation of the knot). Whether a finite-dimensional family $X \subset \Kn(M)$ is preserved by such involution, is the question of existence of an antiholomorphic family on $X$ that acts accordingly on the marked knots $X_{\bullet}$. Whenever it exists, it is unique.

\begin{df}\label{real def}
    Let $\mathcal{X}=(X_{\bullet}, X, \pi, \ev)$ be a family of knots in $M$. The \emph{conjugate family} is the  family $\overline{\mathcal{X}}$ that differs from $\mathcal{X}$ only by choice of orientation. 
    
    A family $\mathcal{X}$ is holomorphic if and only if $\overline{\mathcal{X}}$ is antiholomorphic. an holomorphic family $\mathcal{X}$ is called \emph{real} if there exists an antiholomorphic involution $\sigma \colon X \to X$ such that $\overline{\mathcal{X}}==(X'_{\bullet}, X, \pi', \ev \circ \sigma')$, where $X'_{\bullet}:=X_{\bullet} \times_{\sigma} X$ and $\pi'$ and $\sigma'$ are the natural maps
    \[
    \xymatrix{
    X'_{\bullet} \ar[d]_{\pi'} \ar[r]^{\sigma'} & X_{\bullet} \ar[d]^{\pi} \ar[r]^{\ev} & M\\
    X \ar[r]_{\sigma} & X &
    }
    \]
\end{df}

\begin{ex}\label{non-real ex}
 The family of all circles $\mathcal{U}$ is real and the involution $\sigma$ acts by complex conjugation. A subvariety $X \subseteq U$ gives a real family of circles if and only if it itself is preserved by the complex involution. 
 
 This gives many interesting examples of non-real families. Indeed, choose $H' \subset \CP^4$ to be a non-real hyperplane sufficiently close to $H$ and consider the intersection $H \cap Q$
 
 The resulting family of circles may be described as follows. Identify $S^3$ with the ideal boundary of $H^4$, and let $Y \subset H^4$ be a Veronese surface (that is, a projection of a complex line in the twistor space which is not vertical but is sufficiently close to one). For each point $y \in Y$, take all the hyperbolic planes $\tau \subset H^4$ through $y$ whose tangent planes $T_y\tau \subset T_yH^4$ induce the same complex structure on $T_yH^4$ as $T_yY$. Their ideal boundaries $\partial\tau$ form a surface $\Sigma_Y$ of knots in $\partial H^4$ parametrised by $\CP^1\times Y$. The link between the twistor space of $H^4$ and holomorphic families of knots in $S^3$ is to be investigated in a separate paper.
 
 There is no metric in the round conformal class for which this is the Hitchin's surface of geodesics.
\end{ex}

The family consisting of the fibres of the Hopf fibration (see Example \ref{Hopf}) is a subfamily of the family of great circles (it is non-real, by the way). Its base is a $\CP^1$ that is embedded into the $2$-dimensional quadric $X = H \cap Q$ as one of the generators.

Another important notion is the following. 

\begin{df}\label{flowlike def}
A family $\mathcal{X}$ in $M$ is \emph{flow-like} if for every point $m \in M$ and a vector $v \in T_mM$ there exists a unique knot $\gamma$ from the family $\mathcal{X}$ that passes through $m$ and the velocity  vector of $\gamma$ at $m$ belongs to $\R_{>0}\cdot v$.
\end{df}

Roughly speaking, this means that  given a point and a direction, one can uniquely extend it to a path along a knot from the family (see also Proposition \ref{flow-like prop} for more compact characterisation). A flow-like family is real if and only if the knots corresponding to $(m, v)$ and $(m, -v)$ differ only  by switch of the orientation.

\begin{ex}
The family of great circles is flow-like and real. The family discussed in the Example \ref{non-real ex} is flow-like but not real. Families constructed in Proposition \ref{surfaces of circles} are not flow-like for $d>1$ and can be real or non-real depending on choices.
\end{ex}

Proposition \ref{surfaces of circles} implies that there are many compact holomorphic families of knots in the round $S^3$ with base of complex dimension $2$. Most of them give rise to families that are either not real, or not flow-like. As we will see, there is a reason behind it. First of all, no conformal 3-manifold admits a compact holomorphic family of knots of complex dimension greater than two. Second, the only examples of flow-like real compact families are essentially obtained from the family of great circles.

\subsection{Families obtained from Hermitian line bundles on complex curves}

Now we discuss a different method of constructing holomorphic families of knots that is based on complex geometry and leads to families of knots in $3$-manifolds of more interesting topology. The idea is to come back to the initial example of the Hopf fibration $h \colon S^3 \to \CP^1$ and view it as the spherisation of the tautological line bundle $\O_{\CP^1}(1) \to \CP^1$.

Let $X$ be a K\"ahler manifold with a K\"ahler Riemannian metric $g$ and $p \colon \mathcal{E} \to X$ an holomorphic vector bundle on it equipped with an Hermitian metric $h$. The \emph{Chern connection} is the unique Hermitian connection $\nabla_h$ on $\mathcal{E}$ such that $\nabla_h^{0,1}=\dibar_{\mathcal{E}}$ (\cite[Proposition 3.12]{Vois}). The Chern connection splits the tangent bundle of the total space $\Tot(\E)$ into a direct sum of \emph{vertical} and \emph{horizontal} components 
\begin{equation}\label{Chern split}
T\operatorname{Tot}(\mathcal{E})=\mathcal{V} \oplus \mathcal{H}^{\nabla},
\end{equation}
where $\mathcal{V}=\ker dp$. There are natural isomorphisms $p^*\mathcal{E} \xrightarrow{\sim} \mathcal{V}$ and $p^*TX \xrightarrow{\sim} \mathcal{H}^{\nabla}$. One ends up with a K\"ahler Riemannian metric $g^{\nabla}= p^*(\operatorname{Re} h) \oplus p^*g$ on $T\Tot(\E)$. 

\begin{lemma}\label{line bundles on curves}
Let $(X, I_X)$ be a smooth complex curve and $g_X$ a K\"ahler Riemannian metric on it. Let $p \colon \mathcal{L} \to X$ be an holomorphic line bundle,  $h$  an Hermitian metric on $\mathcal{L}$ and $\nabla$ its Chern connection. Denote by $S^1\mathcal{L}:=\{(x,v) \in \mathcal{L} \ | \ ||v||_h=1\}$  the bundle of unit circles in this metric. Let $\widetilde{g}$ be the restriction of the metric $g^{\nabla}$ on $S^1\mathcal{L}\subset \Tot(\mathcal{L})$ and $\pi:= p|_{S^1\mathcal{L}}$.

Then $\mathcal{X} = (X_{\bullet}=S^1\mathcal{L}, X,  \pi, \id_{S^1\mathcal{L}})$ is an holomorphic family of knots in the three-dimensional Riemannian manifold $M=(S^1\mathcal{L}, \widetilde{g})$.
\end{lemma}

\begin{proof} Since $\nabla$ is Hermitian, the horizontal distribution $\mathcal{H}^{\nabla}$ restricts to a distribution on $S^1\mathcal{L} \subset \Tot(\mathcal{L})$ that we also denote $\mathcal{H}^{\nabla}$. One has $TS^1\mathcal{L}=\mathcal{V} \oplus \mathcal{H}^{\nabla}$ with $\mathcal{V}=\Ker d\pi$.

Let $x \in X$ be a point on the base and $F_x=\pi^{-1}(x)$ a fibre over it. Then the bundle $N_{F_x/\mathcal{L}^1}=N_{\gamma/M}$ can be identified with $\mathcal{H}^{\nabla}|_{F_x}$. The map $\delta_x \colon T_xX \to \Gamma(F_x, N_{F_x/\mathcal{L}^1})$ then identifies with the horizontal lift $v \mapsto \widetilde{v}, \ \nabla\widetilde{v}=0$. In particular,  $\delta_x$ is injective; since the evaluation map is a bijection, this is indeed a family in the sense of Definition \ref{definition of a family}. 

Notice that for a point $y \in F_x$ and an horizontal vector $u \in \mathcal{H}^{\nabla}_y$ the inner product $\dot{\gamma}(y) \times_{\widetilde{g}} u$ is again horizontal, because $\mathcal{V} \perp_{\widetilde{g}} \mathcal{H}^{\nabla}$. The operation $u \mapsto \dot{\gamma}(y) \times_g u$ defines a complex structure operator on $\mathcal{H}^{\nabla}_y$ that we denote by $I_{\nabla}$. This turns $\mathcal{H}^{\nabla}$ into a complex rank one vector bundle over $S^1\mathcal{L}$.

The differential $d\pi$ induces an isomorphism $d\pi \colon \mathcal{H}^{\nabla} \xrightarrow{\sim} \pi^*TX$. This is an isomorphism of real vector bundles of rank $2$. Both of these vector bundles are endowed with complex structure operators and to prove the holomorphicity of family $\mathcal{X}$ one needs to show that this isomorphism commutes with the complex structure:
\[
d\pi \circ I_{\nabla}=I_X \circ d\pi.
\]
The isomorphism $d\pi$ is an isometry with respect to Hermitian metrics $\widetilde{g}|_{\mathcal{H}^{\nabla}}$ and $p^*g_X$. But on a given $2$-dimensional real vector space there exists only two complex structures which are Hermitian with respect to a given positive scalar product. Both of them should be elements of order $4$ in $\U(1)$, and there are only two of such elements. Therefore $d\pi$ sends $I_{\nabla}$ either to $I_X$ or to $-I_X$. The latter is ruled out by the fact that $\nabla^{0,1}=\dibar_{\mathcal{L}}$.
\end{proof}

\begin{rmk}
Such families are neither real nor flow-like.
\end{rmk}

It is worthy to mention a generalisation of this principle for non-compact curves.

\begin{prop}\label{horizontal-conformality}
    Let $(M,g)$ be a Riemannian 3-manifold, and $f \colon M \to U \subset \R^2$ a proper differentiable map. The connected components of its fibres form a complex curve iff there exists an embedding of $U$ into $\C$ which makes $g(df,df) = 0$.
\end{prop}
\begin{proof}
    Let $u \in U$ be a point in the image of $f$, and $v \in T_uU$ a tangent vector. One may consider its horizontal lift, that is, a unique vector field along $f^{-1}(z)$ which is orthogonal to it with respect to $g$ and maps to $v$ by the differential of the projection at each point on $f^{-1}(z)$. Then the fibres form a holomoprhic family iff the 90 degrees rotation of any horizontal lift is an horizontal lift again. This defines a complex structure on $U$, and thus its map to $\C$ for which $df$ is horizontally weakly conformal at each point of $M$. This condition is equivalent to $g(df,df) = 0$, see \cite{Baird} and references therein.
\end{proof}

Combined with the results of \cite{GL}, this allows to construct a lot of non-compact one-dimensional holomorphic families of knots in conformally round $S^3$ whose members are not circles.

\section{Holomorphic families of knots and CR-geometry}\label{CR-section}
In this section, we obtain a nice criterion for a family of knots to be holomorphic in terms of LeBrun's CR structure on the spherisation of the cotangent bundle of a conformal $3$-manifold. This approach is essentially a reformulation of a paper of Lempert \cite{Lemp} that connects holomorphic families of knots in a conformal $3$-manifold to holomorphic families of transverse negative knots in it's LeBrun's CR twistor space.

\subsection{The velocity map}\label{velocity basics}

Let  $(M, [g])$ be a conformal $3$-manifold and $\mathcal{X}=(X_{\bullet}, X, \pi, \ev)$ a family of knots in it. Let 
\[
\mathbf{S}M:=(T^*M \setminus 0_{T^*M})/(\R^{\times}_{>0})
\]
be the spherisation of the cotangent bundle of $M$ (here $0_{T*M}$ denotes the zero section). It is equipped with a structure of $S^2$-bundle $p \colon \mathbf{S}M \to M$ that is in fact trivial by Stiefel's parallelisability theorem (\cite{Stief, Kirb}). Thus one has a non-canonical diffeomorphism $\mathbf{S}M \simeq S^2 \times M$.

The kernel of the Liouville $1$-form $\lambda$ on $T^*M$ descends to a well-defined corank $1$ distribution $\Lambda$ on $\mathbf{S}M$. Explicitly, if $m \in M$ and $[\alpha] \in (T^*_mM \setminus\{0\})/\R^{\times}_{>0}$, the hyperplane $\Lambda_{m,[\alpha]} \subset T_{m,[\alpha]}\mathbf{S}M$ is described as
\[
\Lambda_{m, [\alpha]}:=\{v \in T_{m, [\alpha]}\mathbf{S}M \ | \alpha(dp(v)) =0\}.
\]

\begin{df}
The \emph{velocity map} of a family is the map $\vel \colon X_{\bullet} \to \mathbf{S}M$ is the map that sends a point on a knot to the plane orthogonal to the knot, i.e. 
\[
\vel \colon (x,t) \mapsto (\gamma_x(t), [g(\dot{\gamma_x}(t), -)]).
\]
\end{df}
If one identifies $\mathbf{S}M$ with the spherisation of $TM$ via $g$ the formula for the velocity map turns to a more pleasant form $(x,t) \mapsto (\gamma_x(t), \dot{\gamma_x}(t))$, hence the name. In this form, the velocity map is independent on the choice of the conformal structure on $M$. The diagram
\[
\xymatrix{
X_{\bullet} \ar[r]^{\vel} \ar[dr]_{\ev} & \mathbf{S}M \ar[d]_p\\
& M
}
\]
evidently commutes.

The following propositions are immediate.

\begin{prop}\label{flow-like prop}
A family $\mathcal{X}=(X_{\bullet}, X, \pi, \ev)$ is flow-like (see Definition \ref{flowlike def}) if and only if $\vel \colon X_{\bullet} \to \mathbf{S}M$ is a bijection.
\end{prop}

\begin{prop}\label{ehresmann prop}
Set $\mathcal{H}:=d\vel^{-1}(\Lambda)$. Then $\mathcal{H} \subset TX_{\bullet}$ is a subbundle that defines an Ehresmann connection of the $S^1$-bundle $\pi \colon X_{\bullet} \to X$.
\end{prop}
\begin{proof}
We can write $\mathcal{H}=\ker \vel^*(\lambda|_{\mathbf{S}M})$. The $1$-form $\vel^*\lambda$ is nowhere vanishing, since 
\[
\lambda\left(d\vel\left(\frac{\di}{\di t}\right)\right)=g\left(\dot{\gamma}_x(t), \dot{\gamma}_x(t)\right) \neq 0.
\]
Moreover, $TX_{\bullet} =\ker d\pi \oplus \mathcal{H}$.
\end{proof}

\subsection{Velocity map and CR geometry}

Recall that an \emph{almost CR structure} on a smooth manifold $N$ is a distribution $L \subset TN \o \C$ satisfying $L \cap \overline{L}=0$. Alternatively, it is given by a distribution $H \subseteq TN$ and a complex structure operator $J \colon H \to H, \ J^2=-1$ (the equivalence of two descriptions is given by $H:=(L+\overline{L})_{\R}, \ L:=H^{1,0}$). An almost CR-structure is called \emph{integrable}, or just a \emph{CR structure}, if $[L,L]\subseteq L$. Equivalently, the Frobenius tensor 
\[
\Lambda^2H \to TN/H, \ \eta_1 \wedge \eta_2 \mapsto [\eta_1, \eta_2] \mod H
\]
is of $(1,1)$-type with respect to $J$. The \emph{CR corank} of an almost CR-structure is defined as
$\dim M-\rk H$.

A(n almost) CR structure of zero CR corank is just a(n almost) complex structure.

A map of almost CR manifolds $f \colon (N_1, L_1) \to (N_2, L_2)$ is called \emph{CR holomorphic} if $df(L_1) \subseteq L_2$. Note that we do not require $N_1$ and $N_2$ to have equal CR codimension.

Let $\mathcal{X}=(X_{\bullet}, X, \pi, \ev)$ be a family of knots in $M$. Let $\mathcal{H}$ be the Ehresmann connection constructed in Proposition \ref{ehresmann prop}. The map $d\pi \colon \mathcal{H} \to \pi^*TX$ is an isomorphism. Thus, we can pull back the complex structure operator $I_X$ to $\mathcal{H}$, obtaining an almost CR structure on $(\mathcal{H}, \widetilde{I_X})$ on $X_{\bullet}$.

This can be restated as follows.
\begin{cor}\label{cr on marked knots}
Let $\mathcal{X}=(X_{\bullet}, X, \pi, \ev)$ be a family of knots in a $3$-dimensional conformal manifold $(M, [g])$. Then any almost complex structure $I_X$ on $X$ yields an almost CR structure $(\mathcal{H}, \widetilde{I_X})$ on $X_{\bullet}$ in such a way that the projection $\pi \colon X_{\bullet} \to X$ is CR holomorphic.
\end{cor}
Note that even if $I_X$ is integrable, $\widetilde{I_X}$ does not have to be.
\subsection{LeBrun's CR-twistor spaces}\label{lebrun subsec}

LeBrun (\cite{LeBrun84}) observed  that a choice of a conformal structure on a $3$-manifold $M$ yields an integrable CR structure on $\mathbf{S}M$ with the CR distribution $\Lambda$. We describe the construction assuming that $M$ is endowed with a fixed Riemannian metric $g$. It follows a posteriori that the construction depends only on the conformal class of $g$.

Let $m \in M$ be a point and $S_m:=p^{-1}(m)$ the spherisation of $T^*_mM$. Let $[\alpha]\in S_m$. The space $\Lambda_{m,[\alpha]}$ fits into an exact sequence
\begin{equation}\label{cr exact triple}
0 \to \ker dp|_{m,[\alpha]} \to \Lambda_{m,[\alpha]} \xrightarrow{dp} \alpha^{\perp} \to 0.
\end{equation}
Here $\alpha^{\perp}$ is the $2$-plane that corresponds to $\ker \alpha$ under the identification $g \colon T^{*}_mM \xrightarrow{\sim} T_mM$.

The space $\ker dp|_{m, [\alpha]}$ is the tangent space to $S_m$ at the point $[\alpha]$. Let $I_0$ be the unique integrable complex structure on the $2$-sphere $S_m$.

The space $\alpha^{\perp}$ is a $2$-plane that inherits a symmetric positive scalar product from $T^*_mM$. Moreover, since $[\alpha]$ determines a normal co-orientation of $\alpha^{\perp}$, it is oriented (here it is crucial that we are considering the spherisation of the cotangent bundle rather than its projectivisation). Therefore it admits a unique orthogonal positive oriented complex structure operator $I_{g, [\alpha]}$. 

In other words, the exact sequence (\ref{cr exact triple}) realises $\Lambda_{m, [\alpha]}$ as an extension of two complex $1$-dimensional vector spaces, namely $(\ker dp|_{m, [\alpha]}, I_0)$ and $(\alpha^{\perp}, I_{g,[\alpha]})$.

Finally, the Levi-Civita connection on $(M,g)$ gives an Ehresmann connection on the bundle $\mathbf{S}M \xrightarrow{p} M$, and thus a  splitting  $s \colon \alpha^{\perp} \to \Lambda_{m, [\alpha]}$ of the exact sequence (\ref{cr exact triple}). Therefore, we can write a complex structure operator on $\Lambda_{m,[\alpha]}$ as
\[
J_{g}:=I_0\oplus s_*I_{g,[\alpha]}.
\]

\begin{thrm}[LeBrun, \cite{LeBrun84}]
The construction above gives a well-defined integrable CR structure $(\Lambda, J_{g})$ on $\mathbf{S}M$. Moreover, it depends only on the conformal class of the metric $g$.
\end{thrm}

The CR $5$-manifold $(\mathbf{S}M, \Lambda, J_g)$ is referred to as \emph{LeBrun's CR twistor space} of a conformal $3$-manifold $(M, [g])$.
\subsection{CR-holomorphicity of velocity maps}

Let $(M, g)$ be a Riemannian $3$-manifold and $\mathcal{X}=(X_{\bullet}, X, \pi, \ev)$ a family of knots on it. Let $I_X$ be a complex structure on $X$.

We saw in the last two subsections that in such a situation one can consider two different CR manifolds. First, the space of marked knots $X_{\bullet}$ possesses an almost CR structure $(\mathcal{H}, \widetilde{I_X})$. Second, there is LeBrun's CR twistor space $(\mathbf{S}M, \Lambda, J_g)$. They are, moreover, related by the velocity map $\vel \colon X_{\bullet} \to \mathbf{S}M$ and from the very definition of $\mathcal{H}$ it follows that $d\vel(\mathcal{H})\subseteq \Lambda$. Therefore, it is natural to ask, when the map $\vel$ is CR holomorphic. As it turns out, this is equivalent to family $\mathcal{X}$ being holomorphic (with respect to $I_X$). This result almost immediately follows from a result of Lempert (\cite{Lemp}) as we know explain.

A knot $\gamma \colon S^1 \to \mathbf{S}M$ is called \emph{transverse} if it is everywhere transverse to $\Lambda$. 
Observe that if $\gamma$ is transverse, then $dp(\dot{\gamma})$ is always non-zero.

Let $\gamma \colon S^1 \to \mathbf{S}M$ be a transverse knot. Let $\gamma(t)=(m, [\alpha])$. Here $[\alpha]$ is an oriented line in $T^*_mM$, or, equivalently, an oriented plane $\ker \alpha \subset T_mM$. Let $(e_1, e_2)$ be a positively oriented basis in $\ker \alpha$. Further, $\gamma$ is \emph{positive} if $(e_1, e_2, dp(\dot{\gamma}))$ constitute a positive oriented basis in $T_mM$ for some (equivalently, any) $t$ and $(e_1,e_2)$.

We say that a family of knots in $\mathbf{S}M$ is transverse (resp., positive) if every knot in this family is transverse (resp., positive).

The following proposition is \cite[Proposition 9.1]{Lemp}.
\begin{prop}\label{liftings criterion}
A knot $\gamma \colon S^1 \to \mathbf{S}M$ is an image of the velocity map of a knot in $M$ if and only if it is transverse and positive.
\end{prop}

\begin{rmk}
The conventions of  the paper \cite{Lemp} are dual to ours, so Lempert calls \emph{Legendrian} what we call transverse and \emph{negative} what we call positive.
\end{rmk}

Let $\gamma \colon S^1 \to \mathbf{S}M$ be a transverse knot. Then its normal bundle $N_{\gamma/\mathbf{S}M}$ can be identified with $\gamma^*\Lambda$. The latter is a complex vector bundle with the complex structure operator $J_g$. This suggests that the space of transverse knots in $\mathbf{S}M$ is an almost complex Fr\'echet manifold in the spirit of subsection \ref{infinite-subsection}. Lempert proved that this is indeed the case, and, moreover, the resulting complex structure is formally integrable (\cite[Corollary 4.4]{Lemp}).

Further, one can imitate the construction of subsection \ref{holomorphic families subsec} and give a definition of a 
\emph{holomorphic family of transverse knots} in $\mathbf{S}M$. Namely, let $\mathcal{Y}=(Y_{\bullet}, Y, \pi_{\mathcal{Y}}, \ev_{\mathcal{Y}})$ be a family of transverse knots in $\mathbf{S}M$ (i.e. the evaluation map $\ev_{\mathcal{Y}}$ now goes from $Y_{\bullet}$ to $\mathbf{S}M$). The transversality condition guarantees that $\mathcal{H}_{\mathcal{Y}}:=d\ev_{\mathcal{Y}}^{-1}(\Lambda)$ is an Ehresmann connection on the fibration $Y_{\bullet} \xrightarrow{\pi_{\mathcal{Y}}} Y$. A tangent vector $v \in T_{y}Y$ admits a unique lifting $\widehat{v} \in \Gamma(\pi_{\mathcal{Y}}^{-1}(y), (\mathcal{H}_{\mathcal{Y}})|_{\pi_{\mathcal{Y}}^{-1}(y)})$. 

Let $I_Y$ be a complex structure on the base $Y$. We say that a family $\mathcal{Y}$ is \emph{holomorphic} (with respect to $I_Y$) if
\[
J_g\cdot d\ev_{\mathcal{Y}}(\widehat{v})=d\ev_{\mathcal{Y}}(\widehat{I_Y \cdot v)}
\]
for every tangent vector $v \in T_yY$.

\begin{thrm}[Lempert]\label{lempert theorem}
Let $(M,[g])$ be a conformal $3$-manifold. Let $\mathcal{X}=(X_{\bullet}, X, \pi, \ev)$ be a family of knots in $M$. Consider the transverse family of knots in $\mathbf{S}M$ given by
\[
\mathcal{X}^{\vel}:=(X_{\bullet}, X, \pi, \vel).
\]
(Here the space of the family is the same, but we replace the evaluation map $\ev \colon X_{\bullet} \to M$ with the velocity map $\vel \colon X_{\bullet} \to \mathbf{S}M$, thus treating it as a family of knots in $\mathbf{S}M$).

Let $I_X$ be a complex structure on $X$. Then $\mathcal{X}$ is an holomorphic family of knots in $(M,[g])$ with respect to $I_X$ if and only if $\mathcal{X}^{\vel}$ is an holomorphic  family of transverse knots in $(\mathbf{S}M, \Lambda, J_g)$ with respect to $I_X$.
\end{thrm}

Lempert phrases this theorem in a slightly different way. He proves (\cite[Corollary 9.6]{Lemp}) that the velocity map determines a biholomorphism between the space of knots $\Kn(M)$ and the space of transverse positive knots in $\mathbf{S}M$. Theorem \ref{lempert theorem} is obtained by restricting to finite-dimensional submanifolds of these Fr\'echet manifolds.

\begin{cor}\label{cr holomoprhic criterion}
Let $\mathcal{X}=(X_{\bullet}, X, \pi, \ev)$ be a family of knots in a conformal $3$-manifold $(M,[g])$. Let $I_X$ be a complex structure operator on $X$. The family $\mathcal{X}$ is holomorphic with respect to $I_X$ if and only if the velocity map
\[
(X_{\bullet}, \mathcal{H}, \widetilde{I_X}) \to (\mathbf{S}M, \Lambda, J_g)
\]
is CR-holomorphic.
\end{cor}

This is an analogue of the Proposition \ref{horizontal-conformality} for holomorphic families with base of dimension greater than one.

\section{Geometry of holomorphic families of knots}\label{geometry-of-families-section}

\subsection{Radon transform and K\"ahler forms}

 Fix a smooth manifold $M$ and a family of knots $\mathcal{X}=(X_{\bullet}, X, \pi, \ev)$ in it. There is an important procedure that relates differential $k$-forms on $M$ with differential $(k-1)$-forms on $X$. We refer to it as the \emph{Radon transform} due to its similarity to the eponymous operation in integral geometry. 
 
 First, recall that if $\eta$  is a $k$-form $\eta$ on $X_{\bullet}$, one can construct $(k-1)$-form $\pi_*\eta$ on $X$  by integration along fibres: 
\[
(\pi_*\eta)_x(v_1,\ldots, v_{k-1})=\int_{F_x} \eta\left(\widetilde{v_1}, \ldots, \widetilde{v_k}, \frac{\di}{\di t}\right)dt,
\]
where  $\frac{\di}{\di t}$ is a positive-oriented trivialisation of $(\ker d\pi)|_{F_x}$ and $\widetilde{v_i} \in \Gamma(\pi^{-1}(x), TX_{\bullet}|_{\pi^{-1}(x)})$ are lifts of $v_i$. 

Equivalently, one can associate to $\eta$ an $(n-k)$-dimensional current $\Delta_{\eta} \colon \phi \mapsto \int_{X_{\bullet}} \phi \wedge \eta$. Then one can check that the pushforward $\pi_*\Delta_{\eta}$ is again a regular current on $X$, so it can be represented by an integration against a form 
\[
\pi_*(\Delta_{\eta}) \colon \psi \mapsto \int_{X} \psi \wedge \pi_{*}\eta.
\]

The operation $\pi_* \colon \Omega^{\bullet}_{X_{\bullet}} \to \Omega^{\bullet-1}_X$ sends closed forms to closed and descends to a map of the de Rham cohomology groups $H^k(X_{\bullet}, \R) \to H^{k-1}(X, \R)$.  If $X$ is compact it coincides with the Gysin morphism. In particular, it sends integral cohomology classes to integral ones.

\begin{df}
The \emph{Radon transform} is the map
\[
R \colon \Omega^{\bullet}_M \to \Omega^{\bullet-1}_X, \ \eta \mapsto \pi_*\ev^*\eta.
\]
\end{df}

\begin{thrm}\label{kaehlerness of families}
Let $(M, g)$ be a smooth $3$-dimensional Riemannian manifold and $\mathcal{X}=(X_{\bullet}, (X, I_X), \pi, \ev)$ an holomorphic family of knots. Then the following holds:
\begin{itemize}
\item[(i)] the complex manifold $(X, I_X)$ is K\"ahler;
\item[(ii)] if $X$ and $M$ are compact, then $X$ is biholomorphic to (the analytification of) a smooth  complex projective variety.
\end{itemize}
\end{thrm}
\begin{proof}

\textit{(i)}. Let $\nu_g$ be the Riemannian volume form on $M$ and $\omega_g:=R(\nu_g)$ its Radon transform. We claim that $\omega_g$ is a K\"ahler form on $(X, I_X)$. If  $v_1$ and $v_2$ are two vectors in $T_xX$, then $\omega_g(v_1,v_2)$ can be computed as 
\[\omega_g(v_1,v_2)=\int_{S^1}\gamma_x^*\nu_g(\delta_x(v_1),\delta_x(v_2), \dot{\gamma_x}) dt.
\]

The form $\omega_g$ is closed and non-degenerate because $\nu_g$ is. What is left to check is that $(u, v) \mapsto \omega_g(u, I_Xv)$ defines a positive symmetric $2$-form  on $X$. Indeed, let $v$ be a non-zero vector in $T_xX$. Then
\[
\omega_g(v, I_Xv)=\int_{F_x} {\nu_g}{\left(\delta(v), {\dot{\gamma_x}}{\times_g}{\delta(v)}, \dot{\gamma_x}\right)}dt.
\]
For any pair of non-zero vectors $u_1, u_2$ in the 3-dimensional Euclidean space the triple $(u, v \times u, v)$ defines a positive-oriented basis. Thus the right-hand side of the equality above is positive.
\\

\textit{(ii)}. Suppose now that $X$ is compact. By rescaling $\nu_g$ if necessary we may assume that the class $[\nu_g] \in H^3(M, \R)$  is integral. Therefore the class $[\omega_g]$ is also integral. Thus, $X$ admits a K\"ahler form with integral cohomology class. By Kodaira's Embedding Theorem, $X$ is biholomorphic to a projective variety (see e.g. \cite[Theorem 7.11]{Vois}).
\end{proof}

\begin{rmk}
Let $\Sigma \subset M$ be a smooth surface. To it one can associate the locus $L_{\Sigma}=\pi(\ev^{*}(\Sigma)) \subset X$ of knots lying on $\Sigma$. This locus is Lagrangian with respect to the symplectic form $\omega_g$, in particular it is never holomorphic.
\end{rmk}

An important corollary of Theorem \ref{kaehlerness of families} is the following.

\begin{lemma}\label{surjectivity of evaluation}
Let $(M,g)$ be a compact  $3$-dimensional Riemannian manifold. Consider an holomorphic family of knots $(X_{\bullet}, X, \pi, \ev)$ in it. Suppose that $X$ is compact and $\dim X >0$. Then the evaluation map $\ev \colon X_{\bullet} \to M$ is surjective.
\end{lemma}
\begin{proof}
Suppose first that $\dim_{\C} X = 1$. Then $\ev \colon X_{\bullet} \to M$ is a smooth map between compact $3$-dimensional manifolds.  To prove that it is surjective, it is enough to check that $\ev^{*}[\nu_g] \neq 0$ where $[\nu_g]$ is the cohomology class of the Riemannian volume form on $(M, g)$. But $\omega_g=R(\nu_g)=\pi_*\ev^*\nu_g$ is a K\"ahler form on the compact complex manifold $X$, so its cohomology class is non-zero. Therefore, $\ev^*[\nu_g]\neq 0$ as well.

Assume $\dim_{\C}X>1$. Since $X$ is projective, one can find a smooth closed holomorphic curve $Y \subset X$. Then the evaluation map is surjective already after restriction to $Y_{\bullet}:=\pi^{-1}(Y)$.
\end{proof}

If one drops the compactness assumption, the statement of Lemma \ref{surjectivity of evaluation} still holds locally.

\begin{prop}\label{local surjectivity}
Let $(M,g)$ be a $3$-dimensional Riemannian manifold and $(X_{\bullet}, X, \pi, \ev)$ an holomorphic family of knots, $\dim_{\C} X > 0$. Let $m \in M$ be a non-critical value of $\ev$. Then there exists a neighbourhood $U$ of $m$ such that $\ev^{-1}(U) \to U$ is surjective.
\end{prop}
\begin{proof}
Let $x_{\bullet}=(x, t)$ be a point in $X_{\bullet}$ such that $\ev(x_{\bullet})=m$. In other words, $m=\gamma_x(t)$. It is sufficient to show that the differential $d\ev|_{x_{\bullet}} \colon T_{x_{\bullet}}X_{\bullet} \to T_mM$ is surjective. Notice that  $\rk d\ev|_{x_{\bullet}} \ge 1$, because it does not vanish restricted on $\ker d\pi|_{x_{\bullet}} = \R\dot{\gamma_x}$. Let $v \in T_xX$ be a tangent vector such that $\delta(v)(t) \neq 0$. Such $v$ exists because we assumed $m$ to be non-critical. Then $\delta(v)(t), \dot{\gamma_x}$ and $\delta(I_Xv)(t)=\dot{\gamma_x} \times_g \delta(v)(t)$ are three linearly independent vectors in $T_mM$ in the image of $d\ev$. Thus $\rk d\ev|_{x_{\bullet}} \ge 3=\dim T_mM$.
\end{proof}

\subsection{Pinned loci}\label{pinned loci subsec}
We say that an holomorphic family of knots $\mathcal{X}=(X_{\bullet}, (X, I_X), \pi, \ev)$ in a Riemannian $3$-manifold $(M,g)$ is \emph{analytic} if there exist real analytic structures on $M$ and $X_{\bullet}$ such that $\ev \colon X_{\bullet} \to M$ and $\pi \colon X_{\bullet} \to X$ are analytic. Starting from now we restrict ourselves to study of analytic families of knots. 

We are not aware of any example of a non-analytic holomorphic family of knots and such an assumption seems to be a minor limitation.

Recall the following result of Reiffen.
\begin{thrm}[Reiffen; \cite{Reiff}, Satz 2']\label{reiffen thrm}
Let $X$ be a compact manifold and $Y \subset X$ a real analytic subvariety such that its smooth locus is a complex submanifold. Then $Y$ is complex analytic.
\end{thrm}

 As usual, for a point $x \in X$ we denote $F_x:=\pi^{-1}(x)$ and $\gamma_x:=\ev|_{F_x}$

\begin{df}
Fix a point $m \in M$. A \emph{locus of knots pinned at $m$} (or, simply, a \emph{pinned locus}) is the set of the form
\[
D_m:=\{x\in X \ | m \in \im \gamma_x\}.
\]
\end{df}

Let $m$ be a point in $M$ and  $\widetilde{D}_m:=\ev^{-1}(m) \subseteq X_{\bullet}$. Then $D_m=\pi(\widetilde{D}_m)$. 

Similarly, for a finite collection of pairwise distinct points $\mathbf{m}=\{m_1, \ldots, m_k\}$ in $M$ one can define the \emph{$k$-pinned locus} $D_{\mathbf{m}}=D_{m_1,\ldots, m_k}$ as the set of knots $[\gamma]\in X$ that pass simultaneously through $m_1, \ldots, m_k$. In other words,
\[
D_{\mathbf{m}}= \bigcap_{m \in \mathbf{m}} D_m.
\]

Clearly, if $\mathbf{m}' \subseteq \mathbf{m}$ then $D_{\mathbf{m}} \subseteq D_{\mathbf{m}'}.$

Since $\widetilde{D}_m$ is a fibre of the analytic map $\ev$ and $D_m$ is its image under the proper analytic map $\pi$, they both are analytic. Let $D^{\circ}_m \subseteq D_m$ (respectively, $\widetilde{D}^{\circ}_m$) be their smooth loci.

Observe that  $\widetilde{D}_m \to D_m$ is a bijection and over a dense open analytic subset $\widetilde{D}^{\circ}_m \to D^{\circ}_m$ is a real analytic isomorphism. Its inverse gives a section of the $S^1$-bundle $(D_m)_{\bullet}:=\pi^{-1}(D_m) \to D_m$.

Recall, that the complex structure operator $I_X$ on $X$ determines an almost CR structure $(\mathcal{H}, \widetilde{I_X})$ on $X_{\bullet}$, where the distribution $\mathcal{H}$ is a certain Ehresmann connection on the fibre bundle $X_{\bullet} \to X$ (Corollary \ref{cr on marked knots}).

\begin{lemma}\label{tambovian properties}
The following holds:
\begin{enumerate}
\item \label{holomorphic} $\widetilde{D}^{\circ}_m$ is a CR holomorphic submanifold of $(X_{\bullet}, \mathcal{H}, \widetilde{I_X})$ of CR corank zero. This means that $T\widetilde{D}^{\circ}_m$ is contained in the restriction of $\mathcal{H}$ and preserved by $\widetilde{I_X}$. Moreover, the restriction of $\widetilde{I_X}$ to $T\widetilde{D}^{\circ}_m$ defines an integrable complex structure on $\widetilde{D}^{\circ}_m$ and the projection $\widetilde{D}^{\circ}_m \to \pi(\widetilde{D}^{\circ}_m)\subseteq X$ is holomorphic;
\item \label{analytic} $D_m$ is a closed complex analytic subvariety of $X$;
\item \label{dimension}if $m \in M$ is not a critical value of $\ev$, then $\dim_{\C}D_m=\dim_{\C}X -1$, i.e. $D_m$ is a divisor;
\item \label{homology class} if $X$ is compact of complex dimension $n$ and $m,  m'$ are two non-critical values of $\ev$, then $[D_{m}]=[D_{m'}]$ in $H_{2n-2}(X, \Q)$.
\end{enumerate}

\end{lemma}
\begin{proof}
(\ref{holomorphic}). Since $\widetilde{D}^{\circ}_m$ is an open dense subset of $\ev^{-1}(m)$, it is tangent to $\ker d\ev$. The latter is contained in $\mathcal{H}$. 

Let $x_{\bullet}=(x,t)$ be point in $\widetilde{D}^{\circ}_m$ such that its image $x=\pi(x_{\bullet})$ is in $D^{\circ}_m$. For a vector $v \in T_xX$ denote by $\widetilde{v}$ its horizontal lift via the connection $\mathcal{H}$. Then $\widetilde{v}_{x_{\bullet}}$ is tangent to $\widetilde{D}^{\circ}_m$ if and only if it belongs to $\ker d\ev$, that is $\delta(v)=0$. In this case, $\delta(I_X\cdot v)$ also vanishes, thus $\widetilde{I_X}\cdot \widetilde{v}_{x_{\bullet}}$ is again tangent to  $\widetilde{D}^{\circ}_m$.
\\

(\ref{analytic}). The smooth locus of $D_m$ is an holomorphic submanifold of $X$ by (\ref{holomorphic}). Therefore, $D_m$ is complex analytic by Reiffen's Theorem (Theorem \ref{reiffen thrm}).
\\

(\ref{dimension}). If $m$ is a non-critical value of $\ev$, then $\dim \widetilde{D}_m=\dim X_{\bullet}-\dim M$ (recall that $\ev$ is locally surjective by Proposition \ref{local surjectivity}). At the same time,
\[
\dim_{\R} D_m=\dim_{\R} \widetilde{D}_m=\dim X_{\bullet}-\dim M=\dim_{\R}X+1-3=\dim_{\R}X-2.
\]

(\ref{homology class}). First of all $[\widetilde{D}_m]=[\widetilde{D}_{m'}]$ since these are two smooth fibers of a map between connected compact orientable manifolds, so they are both Poincar\'e dual to the pullback of the cohomology class of a point. At the same time, $[D_m]=\pi_*[\widetilde{D}_m]$ and $[D_{m'}]=\pi_*[\widetilde{D}_{m'}]$.
\end{proof}
\begin{rmk}
The assumption that $m$ is non-critical is crucial in the item (\ref{dimension}). Indeed, suppose that every knot $\gamma_x$ in the family $X$ passes through a point $m_0 \in M$. Then $D_{m_0}=X$.
\end{rmk}

\begin{rmk}
 Lemma \ref{tambovian properties} can be generalised directly to $k$-pinned loci. We left as an exercise to the reader to check that for a generic choice of distinct $k$-points $\mathbf{m}=\{m_1, \ldots, m_k\}$ in $M$ the $k$-pinned locus $D_{\mathbf{m}}$ is a complex subvariety of $X$ of complex codimension $k$ whose homology class does not depend on the choice of $\mathbf{m}$.
\end{rmk}

The following  follows immediately from Lemma \ref{tambovian properties} and Corollary \ref{cr holomoprhic criterion}.

\begin{cor}\label{velocity as rational map}
Let $m \in M$ be a point.  Then the  velocity map $\vel \colon X_{\bullet} \to \mathbf{S}M$ restricts to an holomorphic map
\[
\widetilde{D}^{\circ}_m \to S_m,
\]
where the $2$-sphere $S_m=\mathbf{S}(T^{*}_mM)\simeq \CP^1$ is endowed with the standard complex structure.
\end{cor}

Let $(m, \xi)$ be a point in $\mathbf{S}M$, that is, $m \in M$ and $\xi$ is a positive ray in $T^*_mM$. Let $\widetilde{D}_{m, \xi}:=\vel^{-1}(m, \xi)$ and $D_{m, \xi}:=\pi(\widetilde{D}_{m, \xi})$. Explicitly, $D_{m,\xi}$ can be described as the set of knots in our family passing through a given point $m$ with the prescribed tangent vector at this point.

\begin{prop}\label{fibres properties}
Suppose that $(m , \xi)$ is a non-critical value of $\vel$ and $m$ is a non-critical value of $m'$. Then
\begin{enumerate}
\item \label{holomorphic limit} $\widetilde{D}_{m,\xi}\cap\widetilde{D}^{\circ}_m$ is a closed  codimension $1$ complex analytic subvariety of $\widetilde{D}^{\circ}_m$;
\item \label{analytic limit} $D_{m, \xi}$ is a closed   codimension $1$ complex analytic subvariety of $D_m$;
\end{enumerate}
\end{prop}
\begin{proof}
(\ref{holomorphic limit}). By Corollary \ref{velocity as rational map} it is a fibre of an holomorphic map to $S_m \simeq \CP^1$;
\\

(\ref{analytic limit}). Arguments are similar to those in Lemma \ref{tambovian properties}.
\end{proof}

\subsection{The dimension bound}\label{at most two subsec}

\begin{thrm}\label{at most two}
Let $(M, g)$ be a Riemannian $3$-manifold and $\mathcal{X}=(X, X_{\bullet}, \pi, \ev)$ a  holomorphic family of knots in it. If $X$ is compact, then $\dim X \le 2$.
\end{thrm}

The idea of the proof is the following. We show that if $X$ is compact then the loci $D_{m, \xi}$ are either empty or zero-dimensional. At the same time, for a generic choice of $(m, \xi)$ one has 
\[
\dim_{\C} D_{m, \xi}= \dim_{\C} D_m-1 - 1= \dim_{\C} X-2
\]
by \textit{(ii)} of Proposition \ref{fibres properties} and \textit{(iii)} of Lemma \ref{tambovian properties}. Thus, one concludes that $\dim_{\C} X \le 2$.

Recall the following simple fact from K\"ahler geometry.

\begin{lemma}\label{fs non-zero}
Let $X$ be a compact K\"ahler manifold and $f \colon X \to \CP^1$ an holomorphic map. Let $Y \subseteq X$ be a closed analytic subvariety and $\omega_{FS}$ be the Fubini--Study form on $\CP^1$. Then $f|_{Y}$ is locally constant  if and only if $[(f^*{\omega_{FS}})|_{Y}]=0$ in $H^2(Y, \Q)$
\end{lemma}
\begin{proof}
Let $k=\dim_{\C} Y$. Choose a K\"ahler form $\omega$ on $X$.

Suppose that $f|_{Y}$ is not locally constant. Then $(f^{*}\omega_{FS})|_Y$ is a non-negative $(1,1)$-form on $Y$ which is positive in at least one point. Therefore,
\[
\int_Y f^*\omega_{FS} \wedge \omega^{k-1} >0,
\]
implying that $f^*\omega_{FS}$ is not exact on $Y$.
\end{proof}

\begin{proof}[Proof of Theorem \ref{at most two}]
Let $m$ be generic enough, so that $D_m$ is smooth. We may assume that $\dim_{\C}X \ge 2$. Therefore $\dim_{\C} D_m \ge 1$ and the evaluation map $\ev \colon \pi^{-1}(D_m) \to M$ is surjective (Lemma \ref{surjectivity of evaluation}). The velocity map restricts to a surjective  map $\vel \colon D_m =\widetilde{D}_m \to S_m=\CP^1$ which is holomorphic by Corollary \ref{velocity as rational map}. 

A generic knot $\gamma \in D_m$ is not contained inside the critical locus of $\ev|_{\pi^{-1}(D_m)}$, i.e. $D_{m,z}=D_m \cap D_z$ is smooth for generic $z \in \gamma$. Notice that $[D_{m,z_1}]$ and $[D_{m, z_2}]$ are homologous inside $D_m$ for different  $z_1, \ z_2 \in \gamma$. Moreover, we can assume that $D_{m, \xi}$ is smooth for $\xi=\dot{\gamma}(m)$ and the same is true for $D_{m, -\xi}$. 

We use notation $D_{m, \pm \xi}:=D_{m, \xi} \cup D_{m, -\xi}$.

The idea of the proof is the following. If we take $z \in \gamma$ and tend it to $m$, the subvarieties $D_{m, z}$ tend to $D_{m, \xi}$. This implies that $\vel^*\omega_{FS}$ integrates by zero over $D_{m,z}$, forcing it to be contained in a union of fibres of the velocity map.

More precisely, let $\Delta_z$ denote the integration current of $D_{m,z}$ in $D_m$. This is a  closed integral current whose homology class does not depend on $z$. Moreover, since each $D_{m,z}$ is a compact holomorphic submanifold inside compact K\"ahler manifold $D_m$, it is volume minimizing by Wirtinger inequality (see e.g. \cite{Fed}). Therefore, the mass $\mathbf{M}(\Delta_z)$ is constant. 

Further, Federer - Fleming Compactness Theorem (\cite{FF}) guarantees  that there exists a limit $\Delta:=\lim_{z \to m} \Delta_{z}$ that is an integral current in the same homology class.

We claim that $\Delta$ is supported on $D_{m, \xi} \cup D_{m, \pm\xi}$.

Proof of the claim is the following. Let $z_n$ be a sequence of points on $\gamma$ that tend to $m$ and $x_n$ a sequence of points in $D_m$ with $x_j \in D_{m, z_j}$. Each $x_j$ corresponds to a knot $\gamma_{x_j}$ that passes through $m$ and intersects $\gamma$ in at least one more point $z_j$. Suppose that $x$ is a limit point of $(x_n)$ (such necessary exists by compactness). Then the knot $\gamma_x$ passes through $m$ and is tangent to $\gamma$ there. This implies that either $\gamma_x \in D_{m, \xi}$ or $\gamma_x \in D_{m, -\xi}$. 

Choose a point $y \in X \setminus D_{m,\pm\xi} $. By the discussion above and compactness of $D_{m, z_j}$, we can find a small ball $B(y)$ centered at $y$ such that $B(y) \cap D_{m, z}$ is empty for $z \in \gamma$ sufficiently close to $m$. Thus $\Supp(\Delta) \cap B(y)$ is empty. 

Once we concluded that $\Supp(\Delta) \subseteq D_{m, \pm\xi}$, we can apply Lemma \ref{fs non-zero}. Namely, let $\omega:=\vel^*\omega_{FS}$ be the pullback of Fubini--Study form on $S_m \simeq \CP^1$ to $D_m$ via the velocity map. For $z \in \gamma$ we have
\[
\int_{D_{m,z}} \omega=\langle \Delta_z, \omega \rangle=\langle \Delta, \omega \rangle =0.
\]
Here the second equality follows from the fact that $\Delta$ and $\Delta_{z}$ are homologeous and the third one from that $\Delta$ is supported on a fibre of $\vel$.

By Lemma \ref{fs non-zero} this implies that connected components of $D_{m,z}$ are contained in a fibre of $\vel \colon D_m \to S_m$ for every $z \in \gamma$. The intersection of  $D_{m,z}$ and $D_{m, \xi}$ is non-empty, as it contains $\gamma$. Thus, $D_{m,z} \subseteq D_{m,\pm\xi}$. Observe that by Proposition \ref{fibres properties} these are both closed analytic subvarieties of the same dimension, hence $D_{m,z}\subseteq D_{m, \xi}$ is a union of  irreducible components. 

Let $E \subseteq D_{m, \pm\xi}$ be the irreducible component containing ${\gamma}$. Since we assumed that $D_{m, \xi}$ is smooth, such component is unique and coincides with the connected component. As we shown, for every $z \in \gamma$ one has $D_{m, z}=E$.  This means, that every knot from $E$ passes through every point of $\gamma$ as well, i.e. $E=\{\gamma\}$.

Arguing similarly for knots in different connected components of $D_{m, \xi}$, we conclude that each of them is a point, i.e. $\dim D_{m, \xi}=0$.

At the same time, $\dim D_{m, \xi}=\dim X-2$, thus $X$ is a surface.
\end{proof}

\subsection{Conformal class is determined by a surface of knots}

\begin{thrm}\label{two-complex-structures-proposition}
   Let $\mathcal{X} = (X_\bullet,X,\pi,\ev)$ be some four-dimensional family of knots in $M$. Suppose the choice of a complex structure $I$ on $X$ makes it holomorphic for a conformal class $[g]$ on $M$, and the choice of a complex structure $I'$ on $X$ makes it holomorphic for a conformal class $[g']$. Then $I' = \pm I$, and $[g] = [g']$.
\end{thrm}
\begin{proof}
    For a point $m \in M$, the locus $D_m$ is holomorphic in both $I$ and $I'$. Choose a generic  point $\gamma \in X$. Then for each $m \in \gamma$ the point $[\gamma]$  is contained in the holomorphic locus $D_m$. Moreover, these loci do not coincide and are transverse at $\gamma$ for infinitely many points $m \in \gamma$.  Therefore there exists infinitely many $2$-dimensional subspaces $T_{\gamma}D_m \subset T_{\gamma}X$ that represent complex lines in both $I$ and $I'$. The two complex structures operators on $\R^4$ can have infinitely many common complex lines if and only if they are either equal or opposite, hence the first claim.  This implies that $I= \pm I'$ on a dense subset, and hence globally.

 We may assume $I=I'$ by replacing $I'$ with $-I'$ if needed. Now  we are in the situation where the same complex family is holomorphic with respect to two conformal structures $[g]$ and $[g']$.

 Recall that a conformal class $[g]$ on a three-dimensional vector space $V$ canonically defines a complex structure tensor $J_{[g]}$ on its spherisation $\mathbf{S}(V)$. In the case where $V=T_mM$ is the tangent space of a conformal $3$-manifold, its spherisation $S_m$ is tangent to the LeBrun's CR distribution (see subsection \ref{lebrun subsec}) and inherits the very same complex structure $J_{[g]}$.

 The two complex structures $J_{[g]}$ and $J_{[g']}$ coincide if and only if $[g]=[g']$. Indeed, we can always write $g'=g(A \cdot, \cdot)$ where $A$ is a self-adjoint positive operator. In an $g$-orthonormal basis $(e_1, e_2, e_3)$ the operator $A$ is of the form 
 \[
 A= \begin{pmatrix} a_1 &0 &0\\0&a_2 &0\\ 0&0&a_3 \end{pmatrix}
 \]
 with $a_i>0$ and $J_{[g']}$ acts on the tangent spaces 
 \[
 T_{[\R_{>0}\cdot e_i]}\mathbf{S}(V) = (\R \cdot e_i)^{\perp_{[g']}}= \langle e_j, e_k \rangle
 \]
 via $e_j \mapsto \sqrt{\frac{a_j}{a_k}} \cdot e_k$ (we assume that the trivector $e_i \wedge e_k \wedge e_i$ is positive oriented). The condition $J_{[g]}=J_{[g']}$ implies that $\sqrt{\frac{a_j}{a_k}}=1$ for each pair $\{j,k\}\subset \{1,2,3\}$. It follows that $a_1=a_2=a_3=c>0$ , i.e. $g'=c\cdot g$ (to put it more geometrically: two ellipsoides in $\R^3$ are conformally equivalent if and only if they are homothetic).
 
In our situation, each sphere $S_m$ inherits the two complex structures $J_{[g]}$ and $J_{[g']}$. The restriction of the velocity map $\vel \colon (D_m, I)=(D_m, I') \to S_m$ is holomorphic with respect to both of them. As this is a non-constant map between complex manifolds of the same dimension, it is generically unramified cover. This means that on a dense open subset of $S_m$ we have \[
J_{[g]}=d\vel \circ I \circ d\vel^{-1}=d\vel \circ I' \circ d\vel^{-1}=J_{[g']}.
\]

Therefore, $J_{[g]}=J_{[g']}$ and $[g]|_{T_mM}=[g']|_{T_mM}$. Again this holds for a dense subset of points $m \in M$, thus $[g]=[g']$.

\end{proof}

\begin{rmk}
  This is again a peculiarity of the finite-dimensional situation: in the infinite-dimensional space of knots, pinned submanifolds are also defined independently of the conformal structure on $M$, and hence are holomorphic for many complex structures on $\Kn(M)$, which does not yield the structures to coincide. It would be interesting to figure out if there exist formally integrable almost complex structures on $\Kn(M)$ for which any pinned submanifold $D_m$ is a complex hypersurface which do not come from any conformal structure on $M$.
\end{rmk}

\section{Flow-like and real flow-like families}\label{flowlike-section}

In this section we study flow-like and real flow-like families (Definitions \ref{flowlike def} and \ref{real def}).

\subsection{Counting knots through pairs of points}

Let $\mathcal{X}=(X_{\bullet}, X, \pi, \ev)$ be a compact flow-like holomorphic family in a conformal $3$-manifold $(M, [g])$. 

Recall that an algebraic variety is called \emph{rationally connected} if every two points on it can be connected by a chain of rational curves.

\begin{prop}
Let $\mathcal{X}=(X_{\bullet}, X, \pi, \ev)$ be a compact holomorphic flow-like family in $(M, [g])$. Then $X$ is a complex projective surface. For a point $m \in M$, the pinned locus $D_m \subset X$ is a rational curve. Moreover, $X$ is rationally connected.
\end{prop}
\begin{proof}

As we pointed out in Proposition \ref{flow-like prop}, the velocity map $\vel \colon X_{\bullet} \to \mathbf{S}M$ is bijective. Since the family is compact, it is a homeomorphism. Therefore it identifies $X$ with the base of an $S^1$-bundle whose total space in $\mathbf{S}M$, so $\dim_{\R}X=4$.

The restriction of the $S^1$-family $X_{\bullet} \to X$ on the pinned locus $D_m$ admits a canonical section $\widetilde{D}_m=\ev^{-1}(m)$. The velocity map maps it isomorphically to the sphere $\mathbf{S}_m$. Thus $D_m$ has genus $0$.

It easily follows from surjectivity of evaluation map  (Lemma \ref{surjectivity of evaluation}) that every two points on $X$ can be connected by a chain of pinned curves $D_{m_i}$ (in fact it is always sufficient to two such curves).
\end{proof}

\begin{cor}
Under the same assumptions, the divisors $D_m, D_{m'} \subset X$ are rationally equivalent.
\end{cor}
\begin{proof}
Since $X$ is rationally connected, $H^1(X, \O_X)=0$. Thus, rational equivalence of divisors is the same as cohomological equivalence (see e.g. \cite[Section 7.1.3.]{Vois}).
\end{proof}

This Corollary allows us to speak about \emph{pinned divisor class} $D$ that corresponds to an holomorphic line bundle $\O(D)$. The pinned curves are obtained as zero sets of holomorphic sections of this bundle. Notice that the family of pinned curves $D_m$ is parametrised by points $m \in M$, that is, has real dimension $3$. The space $\mathbf{P}(H^0(X, \O(D))$ is therefore has complex dimension at least $2$, and for generic section of $\O(D)$ its zero set is not a pinned curve, but rather its holomorphic deformation.

The main goal for now is to get a bound on the cohomological self-intersection $[D]^2 \in H^4(X, \Z)$. Namely, we prove the following lemma:

\begin{lemma}\label{square bound}
Let $\mathcal{X}=(X_{\bullet}, X, \pi, \ev)$ be a flow-like holomorphic family on $(M, [g])$. Let $D_m \subseteq X$ be a pinned curve. Then $[D_m]^2 = 2$.
\end{lemma}

The square $[D_m]^2$ can be counted as the number of intersections of two generic pinned curves $D_{m_1}$ and $D_{m_2}$. Geometrically, it can be interpreted as the number of knots from the family $\mathcal{X}$ connecting  $m_1$ to $m_2$.  Therefore Lemma \ref{square bound} says that  flow-like families share something in common with classical geometries: every two points can be connected by a preferred path, and there is at most two such paths.

Once we think about the number $[D_m]^2$ as the number of knots connecting two generic points, it is easy to observe that it equals the degree of \emph{bievaluation map}
\[
\ev^{(2)} \colon X_{\bullet} \times_{X} X_{\bullet} \to M \times M.
\]
This map takes a triple $(\gamma_x, t, s)$, where $t$ and $s$ are two points on the fiber $\gamma_x \subset X_{\bullet}$ to $(\gamma_x(t), \gamma_x(s)) \in M \times M$.

Notice that  $\dim X_{\bullet} \times_{X} X_{\bullet}=\dim M \times M = 6$ and there exists a dense open subset of $M\times M$ on which every point has exactly $d=[D_m]^2$ preimages.

The idea is to connect the number $d$ with the geometry of velocity map $\vel$ seeing the latter as degeneration of $\ev^{(2)}$ when the two points $t, \ s$ tend to each other.

In order to do it let us first introduce the construction of \emph{real  blow-up}. This is an operation that is probably well-known, but we did not manage to find its discussion in the literature.

Let $L$ be a smooth compact manifold and $K \subset L$ a smooth closed submanifold. Let $U(K)$ be a small tubular neighbourhood and $\mathring{U}(K):=U(K) \setminus K$. Consider the projectivisation of the normal bundle $\mathbf{P}N_{K/L}$. It is endowed with the tautological line bundle $\O(1)$. Let $P_0\subset \Tot(\O(1))$ be the zero section and $V(P_0)$ its small tubular neighbourhood in $\Tot(\O(1))$. Let $\mathring{V}(P_0)=V(P_0) \setminus P_0$. By Tubular Neighbourhood Theorem there exists a diffeomorphism $\phi \colon \mathring{U}(K) \xrightarrow{\sim} \mathring{V}(P_0)$.

\begin{df}
The \emph{real oriented blow-up} of $L$ in $K$ is the glueing
\[
\Bl_K L:= (L \setminus K) \bigsqcup_{\phi} V(P_0).
\]
\end{df}

The resulting  topological space is a smooth manifold that does not depend on the choice of $\phi$.

It is endowed with a natural surjective map $\tau \colon \Bl_K L \to L$ that is isomorphism on $\operatorname{Int}(\Bl_K L) \xrightarrow{\sim} L \setminus K$ and on $P_0$ acts as the projection $\mathbf{P}N_{K/L} \to K$.

We come back to study the bievaluation map $\ev^{(2)}$. 

Remark that $X_{\bullet}$ is embedded as a diagonal $\Delta_{X_\bullet} \subset X_{\bullet} \times_X X_{\bullet}$ as a hypersurface. Similarly, $M$ is embedded as a diagonal $\Delta_M \subset M \times M$ and $\ev^{(2)}(\Delta_{X_{\bullet}})=\Delta_M$. The normal bundle $\mathcal{N}_{\Delta_M/M\times M}$ is isomorphic to $T^*M$. Using the conformal structure we can identify $\mathbf{P}(T^*M)$ with $\mathbf{P}(TM)=\mathbf{P}M$. The latter is obtained by taking quotienf of $\mathbf{S}M$ by the natural involution $\pm1$.

The key observation is that the map $\ev^{(2)}$ admits a natural lift to the real  blow-up of $M\times M$ along the diagonal.

\begin{prop}
There exists a smooth surjective map $\Phi$ such that the diagram
$$
\xymatrix{
& \Bl_{\Delta_M}(M \times M) \ar[d]^{\tau}\\
X_{\bullet}\times_X X_{\bullet} \ar[r]_{\ev^{(2)}} \ar[ru]^{\Phi}& M\times M
}
$$
commutes. On the diagonal $\Delta_{X_{\bullet}}$ it coincides with the composition map 
\[X_{\bullet} \simeq \Delta_{X_{\bullet}} \xrightarrow{\vel}  \mathbf{S}M \xrightarrow{\sigma} \mathbf{P}M,
\]
where $\sigma$ is the natural projection.
\end{prop}
\begin{proof}
Set-theoretically this map is already determined: it coincides with $\ev^{(2)}$ on the complement of diagonal  $\Delta_{X_{\bullet}}$ and on the diagonal maps $(\gamma_x, t, t) \mapsto (\gamma_x(t),[\dot{\gamma_x}(t)])\in \mathbf{P}M$.

The only thing that is needed is to check smoothness along the  diagonal. By Boman's theorem it is enough to check it on germs of smooth curves.

Let $c \colon {]}{-\epsilon; \epsilon}{[} \to X_{\bullet} \times_{X} X_{\bullet}$ be a germ of a smooth path such that $c(0)$ lies on diagonal $\Delta_{X_{\bullet}}$. We may assume that $c$ is transverse to $\Delta_{X_{\bullet}}$ and does not intersect it outside $c(0)$, otherwise the smoothness follows from smoothness of $\vel$ and $\sigma$. 

Let $c(\lambda)=(\gamma_{\lambda}, t_{\lambda}, s_{\lambda})$. By our assumption, $t_0=s_0=t$. Let 
\[
(p_{\lambda}, q_{\lambda}):=\Phi(c(\lambda))=(\gamma_{\lambda}(t_{\lambda}), \gamma_{\lambda}(s_{\lambda})) \in M \times M.
\]
The isomorphism $\phi \colon \mathring{U}(\Delta_M) \to \mathring{V}(P_0)$ allows us to write down a point $(p,q)$ in a neighbourhood of $\Delta_M$ as a point $(m, [v], r) \in V(P_0)$, where $(m, [v])$ is a point in $\mathbf{P}M$ and $r \in \R\setminus \{0\}$ is the normal parameter. We are allowed to choose $\phi$ in such a way that $m=p$. Then, after applying diffeomorphism $\phi$, we get
\[
\Phi(\gamma_{\lambda}, t_{\lambda}, s_{\lambda})=(p_{\lambda}, [v_{\lambda}], r_{\lambda})
\]
and $r_{\lambda}v_{\lambda}=(s_{\lambda}-t_{\lambda})\dot{\gamma_{\lambda}}(t_{\lambda})+o(s_{\lambda}-t_{\lambda})$. Here $(s_{\lambda}-t_{\lambda})$ is the distance between $s_{\lambda}$ and $t_{\lambda}$ taken with a sign depending on whether the shortest path between them goes along the orientation of the knot, or in the opposite direction.

As $\lambda \to 0$ from either above to below, $r_{\lambda} \to 0$ and $(s_{\lambda}-t_{\lambda}) \to 0$. By L'H\^{o}pital rule,
\[
\lim_{\lambda \to 0}\frac{r_{\lambda}v_{\lambda}}{r_{\lambda}} = \lim_{\lambda \to 0}\frac{(s_{\lambda}-t_{\lambda})\dot{\gamma_{\lambda}}(t_{\lambda})+o(s_{\lambda}-t_{\lambda})}{s_{\lambda}-t_{\lambda}},
\]
thus $\lim v_{0}=[\dot{\gamma_{0}}(t_0)]$ and the map is indeed continious.

The derivative of $\frac{d}{d\lambda}|_{\lambda=0}\Phi(c(\lambda))$ is $\frac{d}{d\lambda}|_{\lambda=0}(s_{\lambda}-t_{\lambda})\cdot \dot{\gamma_{0}}(t)$ is smooth in $\lambda$, and the same is for higher derivatives.
\end{proof}

\begin{proof}[Proof of Lemma \ref{square bound}]
Degree of $\ev^{(2)}$ equals degree of $\Phi$. 

Let $(m,[v])\in P_0$ be the point on the exceptional locus of $\Bl_{\Delta_M}(M \times M)$. It has exactly two preimages under $\Phi$. We need to check that the differential of $\Phi$ is surjective at these preimages. This would mean that that $\Phi$ is a 2-to-1 cover in a neighbourhood of $(m, [v])$. Since the number of preimages under the map $\ev^{(2)}$ is constant on a dense open subset of $M \times M$, this would be enough. 

Let $(\gamma, t, s) \in \Delta_{X_{\bullet}}$ be a point from preimage of $(m, [v])$. The tangent space $T_{(\gamma, t, s)}(X_{\bullet}\times_X X_{\bullet})$ is generated by the susbspace $T_{(\gamma, t, s)}\Delta_{X_{\bullet}}$ and two vectors $\frac{\di \gamma}{\di t}$ and $\frac{\di \gamma}{\di s}$ (there is exactly one linear relation, namely $ \frac{\di \gamma}{\di t}+ \frac{\di \gamma}{\di s} \in T_{(\gamma, t, s)}\Delta_{X_{\bullet}}$).

The differential $d\Phi$ is an isomorphism on the subspace $T_{(\gamma, t, s)}\Delta_{X_{\bullet}}$ since there it coincides with $d\vel$. At the same time, $d\Phi(\frac{\di \gamma}{\di t})$ and $d\Phi(\frac{\di \gamma}{\di s})$ are non-zero because the evaluation map is immersive on $\gamma$. By dimension count $d\Phi$ is surjecive at point $(\gamma, t, s)$.

\end{proof}

\subsection{Hirzebruch surfaces}

Our goal is classification of holomorphic  flow-like and flow-like real families parametrised by compact complex surfaces. It turns out that a base of such family necessary belongs to a certain kin of projective surfaces known as \emph{Hirzebruch surfaces}. We briefly recall their main properties. The primary reference is \cite[Chapter V, Section 4]{BHPV}.

\begin{df}
Let $\O_{\CP^1}(-1)$ denote the anti-tautological bundle on $\CP^1$ and $\O_{\CP^1}(-n)$ its $n$-th tensor power. The \emph{$n$-th Hirzebruch surface} is $\F_n:=\mathbf{P}(\O_{\CP^1} \oplus \O_{\CP^1}(-n))$.
\end{df}
The surface $\F_n$ is a smooth projective surface that is an holomorphic $\CP^1$-bundle over $\CP^1$. All Hirzebruch surfaces are rational. The surface $\F_0$ is isomorphic to $\CP^1 \times \CP^1$. For $n \neq \pm m$ the surfaces $\F_n$ and $\F_m$ are birational, but not biholomorphic. They are homeomorphic as smooth manifolds if and only if $n$ and $m$ have the same parity. The surface $\F_1$ is isomorphic to the blow-up of $\CP^2$ in a point. 

There is a section $\CP^1 \to \Sigma_n$ that is given by sending a point to the fibre of of the line bundle $\O_{\CP^1}(-n)$. Its image $\Sigma$ is an holomorphic rational curve with self-intersection $[\Sigma]^2=-n$. It is the only irreducible holomorphic curve on $\F_n$ with negative self-intersection. 

Since $\F_n$ is rational, $H^1(\F_n, \Z)=H^3(\F_n, \Z)=0$. The group $H^2(\F_n, \Z)$ is generated by the class $f$ of the fibre of the map $\F_n \to \CP^1$ and the class $s$ of $\Sigma$. The intersection pairing in these generators is given by the matrix
\[
\begin{pmatrix} 0 & 1 \\ 1 & -n \end{pmatrix}.
\]

Vice versa, every rational surface with $b_2=2$ is a Hirzebruch surface.

Recall that on a smooth projective surface the intersection number of two distinct complex curves is always non-negative.

\begin{prop}\label{intersection on hirzebruch prop}
Let $\F_n$ be a Hirzebruch surface, and $D \subseteq \F_n$ a smooth curve. Suppose that $0<[D]^2 \le 2$. Then either $n=0$ and $[D]=(f+s)$ or $n=2$ and $[D]=(2f+s)$.
\end{prop}
\begin{proof}
Let $[D]=(af+bs)$. First of all, $a >0$, otherwise $[D]^2=b^2s^2=-b^2n^2$ is negative.

Since $D$ is smooth and distinct from $\Sigma$, the product $[D].[\Sigma]$ is non-negative. Thus
\[
(af+bs)\cdot s=a-nb \ge 0,
\]
which implies $a \ge nb$. 

Observe that $b>0$ (otherwise $[D]^2=0$). Therefore we conclude that $a \ge n$.

At the same time,
\[
[D]^2=(af+bs)^2=2ab-nb^2=(2a-bn)b.
\]

If $[D]^2=1$, then $b=1$ Thus $b = 1$ or $b = 2$. If latter, the right-hand side is divisible by $4$. Thus $b = 1$, and $a=\frac n2+1\geqslant n$. Therefore $n = 0, a = 1$ or $n = 2, a = 2$.
\end{proof}

\subsection{Flow-like and real flow-like families}

In this subsection, we prove our classification of flow-like and real flow-like families of knots.

\begin{thrm}\label{flowlike thrm}
    Suppose $\mathcal{X}$ is a flow-like  compact holomorphic family of knots in surface of knots in $(M, [g])$. Then the base $X$ of $\mathcal{X}$ is biholomorphic to either $\CP^1 \times \CP^1$ and $M$ is diffeomorphic to a lens space.
\end{thrm}
\begin{proof} 
    The flow-like condition guarantees that the velocity map $\vel \colon X_{\bullet} \to \mathbf{S}M$ is a diffeomorphism.

    Each pinned curve $D_m$ is mapped isomorphically to $S_m \simeq \CP^1$ via the velocity map. Thus, all the pinned curves are rational. This implies that $X$ is rationally connected, and therefore rational by the Enriques--Kodaira classification \cite[Chapter VI]{BHPV}. In particular, $\pi_1(X)$ is trivial. 
    
    The exact sequence of homotopy groups for the fibration $X_\bullet \to X$ reads: 
    \begin{equation}\label{exact}
    0 \to \pi_2(X_\bullet) \to \pi_2(X) \arr{c_1} \pi_1(S^1) \to \pi_1(X_\bullet) \to 1.
    \end{equation}
We treat separately two cases, depending on whether $c_1=0$ or not.
\\

\textbf{Case 1.} 
    If the fibration admits a trivialisation, $c_1 = 0$ and $\pi_1(X_\bullet) = \Z$. Flow-likeness implies that $X_\bullet \cong M \times S^2$, hence $\pi_1(M) = \Z$ and $M = S^2 \times S^1$ (see e.g. \cite[Proposition 3.4]{Hatcher}). Let $D_m \subset X$ be a pinned curve and $(D_m)_{\bullet}$ its preimage in $X_{\bullet}$. Then $(D_m)_{\bullet}$ is also isomorphic to $S^2 \times S^1$ and $\ev \colon (D_m)_{\bullet}\to M$ is surjective. It is moreover of non-zero degree, since the Radon transform of the volume form on $M$ integrates positively over $D_m$. At the same time, the sphere $\widetilde{D_m} \subset (D_m)_{\bullet}$ that corresponds to the natural section $\gamma \mapsto \gamma^{-1}(m)$ is mapped to a point. Thus, $\ev_{*}$ is not surjective on $H^2$, and by Poincar\'e duality $\ev^*$ is not injective on $H^2$. Together with formula
    \[
    f_*f^*\alpha = \deg(f)\alpha
    \]
    for any $\alpha \in H^{\bullet}(M, \Q)$ this leads to a contradiction.
    \\

    \textbf{Case 2.} In this case, $H_2(X,\Z) \cong \pi_2(X) \cong \pi_2(S^2 \times S^1 \times S^2) = \Z^2$. A rational surface with $b_2 = 2$ is a Hirzebruch surface, see \cite[Chapter VI]{BHPV}.

    If the fibration has degree $d \neq 0$, it follows that $\pi_1(M) \cong \Z/d$, and thus $M$ is a lens space by Perelman--Hamilton solution of Thurston Geometrisation Conjecture. In this case, $H_2(X,\Z) \cong \pi_2(X) \cong \pi_2(X_\bullet) \oplus \ker c_1 \cong \pi_2(M \times S^2) \oplus \Z \cong \Z^2$. Thus $b_2(X) = 2$ regardless. 

    For a generic $m$, the pinned curve $D_m$ is a smooth curve with self-intersection equal to $2$ (Lemma \ref{square bound}). By Proposition \ref{intersection on hirzebruch prop} there are only two cases: either $X \cong \F_0 \cong \CP^1 \times \CP^1$, or $X \cong \F_2$. 

    Suppose $X \cong \F_1$. Then $[D_m]=(2f+s)$ and $[\Sigma].[D_m] =s(2f+s)=0$. At the same time, this number equals the class of the fibre  $\ev^{-1}(m)$ for the evaluation map of the family $\Sigma_{\bullet} \to M$. This again contradicts this map having non-zero degree.
\end{proof}
    
\begin{thrm}\label{flowlike real thrm}
    Suppose $X$ is a flow-like real holomorphic surface of knots in $(M,[g])$. Then $M$ is diffeomorphic to either $S^3$ or $\R\mathbf{P}^3$, $X$ is biholomorphic to $\CP^1 \times \CP^1$, $[g]$ is the conformal class of the round metric, and $X$ is a surface of geodesics for some round metric in this class.
\end{thrm}
\begin{proof}
    From the exact sequence (\ref{exact}) one sees that the fundamental group $\pi_1(M) \cong \pi_1(X_\bullet)$ is generated by that of the (image of the) fibre, that is, the class of the knot. In particular, if $X$ is real, this class equals minus itself, as the knots passing along the vectors $v$ and $-v$ at one point $m$ are homotopic and differ only by the orientation. This implies that either $\pi_1(M) = 0$ (and thus $M \cong S^3$) or $\pi_1(M) = \Z/2$ (and $M \cong \R\mathbf{P}^3$).
    
    We know that $X \cong \CP^1 \times \CP^1$ or $X \cong \Sigma_2$. Any fixed-point-free involution on the Hirzebruch surface $\Sigma_2$ preserves its negative-square section $\Gamma$, hence in the latter case $\mathcal{X}|_\Gamma$ is also a real family. The number of knots in this family which pass through a generic point $m \in M$ is given by the intersection number $[\Gamma].[D_m] = s(2f+s) = 0$, which contradicts Proposition \ref{surjectivity of evaluation}.

    Thus $X \cong \CP^1 \times \CP^1$ is the only possible case. Consider the embedding $\varphi \colon X \to \CP^3$ via the linear system of the pinned curves. It realises $X$ as a smooth real quadric without real points inside a complex projective space with the standard real structure. Let $(V,q)$ be the corresponding real quadratic space ($q$ is only defined up to a scalar multiple, and thus may be considered positive definite). Define the map $\delta \colon M \to \mathbf{P}(V^*)$ by sending $m \in M$ to the real plane which cuts out the corresponding pinned curve $D_m \subset X$. It is a diffeomorphism for $M \cong \R\mathbf{P}^3$ and a double cover for $M \cong S^3$. If $M \supset \gamma \in X$ is a knot from our family, $\delta(\gamma) \subset \mathbf{P}(V^*)$ is the set of all real planes in $V$ which pass through the point $\varphi(\gamma)$, that is, a projective line in $\mathbf{P}(V^*)$. The quadratic form $q$ induces a metric $\tilde{q}$ on $\mathbf{P}(V^*)$ (resp. its double cover) for which the family of straight lines (resp. great circles) is a surface of knots. The standard construction identifies this surface with $\mathbf{P}\{q(v,v)=0\} = \varphi(X)$, and, if the orientation on $V$ is chosen correctly, the line $\delta(\gamma)$ goes to the point $\gamma$. By Theorem \ref{two-complex-structures-proposition}, $\delta^*\tilde{q} \in [g]$.    
\end{proof}

\begin{rmk}
    LeBrun proved \cite[Section 4]{LeBrun84} that one can reconstruct a conformal structure on a Riemannian threefold from certain data on its CR twistor space---namely, a CR structure, its fibration by holomorphic $\CP^1$s, an anti-CR holomorphic involution preserving the fibration, and a distribution which LeBrun calls the ``CR contact structure'' (it is an odd-dimensional counterpart of the holomorphic contact structure on the Atiyah--Hitchin--Singer twistor space). In principle, one may employ this theorem to prove the above result. In particular, it explains why the assumption of the surface being preserved by the real structure on the space of knots is crucial.
\end{rmk}

\begin{rmk}\label{singular rmk}
    If our definition of an holomorphic family were slightly more liberal, it would not be possible to exclude $\F_2$ the way we did: one may happen that $\Sigma \subset \F_2$ is contracted into a point. There exists an embedding of $\F_2 \setminus \Sigma$ into the space of round circles in $S^3$ (it is the once-pinned locus inside the non-compact complex threefold of all circles, or, equivalently, the space of straight lines in $\R^3$). If one allows circles to degenerate into points, this would be the map from the Hirzebruch surface to a space of knots which contracts $\Sigma$ into a point, much like the standard map $\F_2 \to \CP^3$ whose image is a quadratic cone does.
\end{rmk}

We end with a natural question that remains unanswered.

\begin{q}
Suppose that $(M, [g])$ admits a compact holomorphic family of knots of complex dimension $2$. Does it mean that $M$ is a sphere or $\R \mathbb{P}^3$ and $[g]$ is the class of conformal metric? More generally, are there other examples of compact complex surfaces of knots except for those constructed in Proposition \ref{surfaces of circles}? 
\end{q}

\vspace{-2mm}

\bibliography{references.bib}{}
\bibliographystyle{alpha}

\noindent {\sc Rodion D\'eev \\
Universit\'e libre de Bruxelles, Boulevard du Triomphe, CP262, 1050, Ixelles, Belgium} \\
\tt rodiondeev@ulb.be
\noindent {\sc Vasily Rogov \\
Universit\'e libre de Bruxelles, Boulevard du Triomphe, CP262, 1050, Ixelles, Belgium} \\
\tt vasirog@gmail.com

\end{document}